\documentclass[uplatex,dvipdfmx,11pt]{article}
\usepackage[margin=25mm]{geometry}

\usepackage[dvipdfmx]{graphicx}

\usepackage{amssymb}
\usepackage{amsmath}
\usepackage{amsfonts}
\usepackage{mathtools}
\usepackage{amsrefs}
\usepackage{mathrsfs}
\usepackage[normalem]{ulem}

\usepackage{color}
	\newcommand{\red}[1]{\textcolor{red}{#1}}

\usepackage{enumerate}

\usepackage{amsthm}
	\theoremstyle{plain}
		\newtheorem{theo}{Theorem}[section]
		\newtheorem{lem}[theo]{Lemma}
		\newtheorem{prop}[theo]{Proposition}
		\newtheorem{defi}[theo]{Definition}
		\newtheorem{cor}[theo]{Corollary}
	\theoremstyle{definition}	
		\newtheorem{remark}[theo]{Remark}
\usepackage{latexsym}

\makeatletter
\@addtoreset{equation}{section}

\makeatother

\usepackage[hang,small,bf]{caption}
\usepackage[subrefformat=parens]{subcaption}
\DefineSimpleKey{bib}{archivePrefix}{}
\BibSpec{arXiv}{%
  +{}{\PrintAuthors}{author}
  +{,}{ \textit}{title}
  +{}{ \parenthesize}{date}
  +{,}{ arXiv: }{eprint}
}
\newcommand{\relmiddle}[1]{\mathrel{}\middle#1\mathrel{}}

\newcommand{\zzz}{\mathbb{Z}^3}
\newcommand{\zd}{\mathbb{Z}^d}
\newcommand{\rrr}{\mathbb{R}^3}
\newcommand{\prob}{\mathbf{P}}
\newcommand{\ust}{\mathcal{U}}

\newcommand{\gamfty}{\gamma_\infty}
\newcommand{\Bx}[1]{B_{x_#1}}
\newcommand{\tS}[1]{t_S(a_{#1})}

\newcommand{\tQ}{\widetilde{Q}}
\newcommand{\vep}{\varepsilon}

\newcommand{\ie}{\textit{i}.\textit{e}.\ }

\newcommand{\Reff}[2]{R_{\mathrm{eff}}(#1,#2)}

\begin{document}
\title{\bf {\large Volume and heat kernel fluctuations for the three-dimensional uniform spanning tree}}
\author{Daisuke Shiraishi   \and  Satomi Watanabe}
\date{Department of Advanced Mathematical Sciences \\ Graduate School of Informatics \\ Kyoto University}

\maketitle
\begin{abstract}
Let ${\cal U}$ be the uniform spanning tree on $\mathbb{Z}^{3}$. We show the occurrence of log-logarithmic fluctuations around the leading order for the volume of intrinsic balls in ${\cal U}$. As an application, we obtain similar fluctuations for the quenched heat kernel of the simple random walk on ${\cal U}$. 
\end{abstract}

\section{Introduction}\label{sec:1}

The aim of this article is to demonstrate an oscillatory phenomenon for the 
heat kernel of the simple random on the three-dimensional uniform spanning tree. To introduce this model of interest here, we start by recalling that Pemantle \cite{P91} proved 
that if $\mathbb{Z}^{d}$ is exhausted by a sequence of finite subgraphs $G_{n}$, then the uniform spanning tree measures on $G_{n}$ weakly converge to some measure supported on spanning forests of $\mathbb{Z}^{d}$. 
The corresponding random graph is called the uniform spanning forest on $\mathbb{Z}^{d}$.  Pemantle \cite{P91} also showed that the uniform spanning forest is a single tree almost surely if $d \le 4$, 
in which case we call it the uniform spanning tree on $\mathbb{Z}^{d}$, 
while it consists of infinitely many trees when $d \ge 5$.  Since 
their introduction, uniform spanning forests have been studied extensively and played an important role in the progress of probability theory, due to connections to various areas such as  loop-erased random walk \cite{Lawler, P91, W96},  electrical networks \cite{BLPS01, Bur-Pem, Kir}, domino tiling \cite{Bur-Pem, Ken}, the random cluster model \cite{Gri, H}, random interlacements \cite{Hut, Szn} and for $d=2$, conformally invariant scaling limits 
\cite{BCK-1, BDW, Hol-Sun, LSW, Sch}. 
Let us mention that the uniform spanning tree is often considered in the same class as various critical statistical physics models since it shares similar properties such as fractal scaling limits with non-trivial scaling exponents and it is one of the few such models for which rigorous results have been proved even for the three-dimensional case \cite{ACHS20, K, LSarx}, which is typically the most difficult case to study. 

Motivating the study of the simple random walk on the uniform spanning forest on $\zd$ is that such a process captures the geometric and spectral properties of  the forest 
and how these depends on the dimension $d$. In particular, 
the random walk displays mean-field behavior for $d \ge 4$, with a logarithmic correction in four dimensions \cite{Hal-Hut, Hut-2}. On the other hand, different (nontrivial) exponents describe the asymptotic behavior of several quantities such as transition density (heat kernel), exit time and mean-square displacement of the random walk  below four dimensions \cite{ACHS20, BM}. 
At least, this is confirmed for $d=2$ and it is strongly believed that this is the case for $d=3$ (see Remark \ref{remarkd=3}).

We now introduce the notation that we need to state our main result of heat kernel fluctuations. 
Let ${\cal U}$ be the uniform spanning tree on $\mathbb{Z}^{3}$. 
We write  $p_{n}^{\cal U} (x,y)$ for the transition density (heat kernel) of the simple random walk on $\ust$, see Section 2.2 for its precise definition. We also let  $\beta \in (1, 5/3]$ be the growth exponent that governs the time-space scaling of 
the three-dimensional loop-erased random walk, which coincides with the Hausdorff dimension of the scaling limit of the three-dimensional loop-erased random walk \cite{S18, S2}, 
see Section 2.1 for details. 
\begin{remark}
	Numerical estimates suggest that $\beta=1.624\cdots$ (see \cite{W10}).
\end{remark}
Our main theorem then demonstrates heat kernel fluctuations of log-logarithmic magnitude as follows.

\begin{theo}\label{main-thm}
There exist deterministic constants $a_{1}, a_{2} > 0$ such that one has  
\begin{equation}\label{main-1}
\liminf_{n\to\infty}  \,  (\log\log{n})^{a_{1}} n^{\frac{3}{3+\beta}}p_{2n}^{\cal U} (0,0)=0, 
\end{equation}
and also
\begin{equation}\label{main-2}
\limsup_{n\to\infty} \, (\log \log n)^{- a_{2}}n^{\frac{3}{3+\beta}}p^{\cal U}_{2n}(0,0)=\infty ,
\end{equation} 
almost surely.
\end{theo}

\

Similar heat kernel fluctuations have been established for Galton-Watson trees \cite{BK06, Cro-Kum} and the uniform spanning tree on $\mathbb{Z}^{2}$ \cite{BCK21}. We will describe some key differences between these models below, but common ingredients in the proofs of such results are corresponding volume fluctuations. In fact, the idea of proof of Theorem \ref{main-thm} is similar to that of \cite{BCK21}*{Corollary 1.2}. Specificaly, in order to prove Theorem \ref{main-thm}, the crucial step is to show that the volume of intrinsic balls (with respect to the graph distance) of ${\cal U}$ also enjoys log-log fluctuations. To be more precise, let $B_{{\cal U}} (0, r)$ be the intrinsic ball in ${\cal U}$ of radius $r$ centered at the origin. Then we have the following volume fluctuations. 

 \begin{theo}\label{main-thm-2}
There exist deterministic constants $a_{3}, a_{4} > 0$ such that one has  
\begin{equation}\label{main-2-1}
\liminf_{r\to\infty}  \,  (\log \log r)^{a_{3}} \,  r^{-\frac{3}{\beta}} \, | B_{{\cal U}} (0, r) |=0, 
\end{equation}
and also
\begin{equation}\label{main-2-2}
\limsup_{r\to\infty} \, (\log \log r)^{- a_{4}} \, r^{-\frac{3}{\beta}} \, | B_{{\cal U}} (0, r) |=\infty 
\end{equation} 
almost surely. Here $|A|$ stands for the cardinality of $A$.
\end{theo}

\begin{remark}
It was 
already proved in \cite{ACHS20}*{Theorem 1.6} that there exist deterministic constants $b_{1}, b_{2}, b_{3}, b_{4} > 0$  and $c_{1}, c_{2} >0$ such that with probability one
\begin{equation}\label{spectral-dim}
c_{1}  n^{-\frac{3}{3 + \beta}}  (\log \log n)^{-b_{1}} \le p^{\cal U}_{2n}(0,0) \le c_{2}  n^{-\frac{3}{3 + \beta}} (\log \log n)^{b_{2}} 
\end{equation}
for large $n$, and also 
\begin{equation}\label{volume-dim}
c_{1}  r^{\frac{3}{\beta}}  (\log \log r)^{-b_{3}} \le  | B_{{\cal U}} (0, r) |  \le c_{2}  r^{\frac{3}{\beta}}  (\log \log r)^{b_{4}}
\end{equation}
for large $r$. Determining the optimal exponents for $a_{i}$ in Theorems \ref{main-thm} and \ref{main-thm-2} seems difficult. 
\end{remark}

\begin{figure}[htb]
\begin{center}
\includegraphics[scale=0.5]{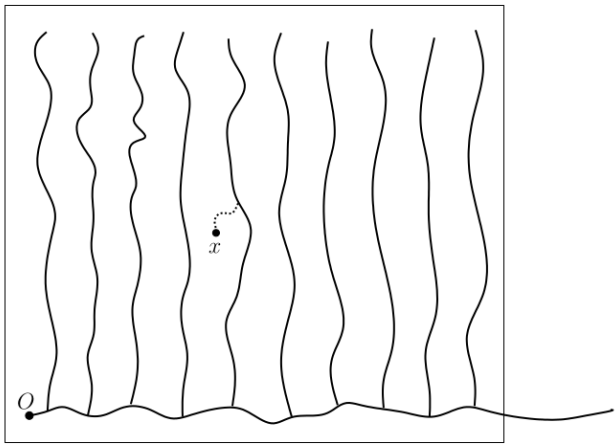}
\caption{Illustration for the comb configuration. The horizontal solid curve stands for the unique infinite path in ${\cal U}$ started at the origin. We force it to keep going to the right with no big backtracking. We also make each vertical solid branch keep going down. For another point $x$, as the dotted curve illustrates, the branch between $x$ and a solid curve has a small length. As a result, if the Euclidean metric between the origin and $x$ is not small, the intrinsic metric from the origin to $x$ is unusually small.}\label{comb}
\end{center}
\end{figure}

\begin{figure}[htb]
\begin{center}
\includegraphics[scale=0.5]{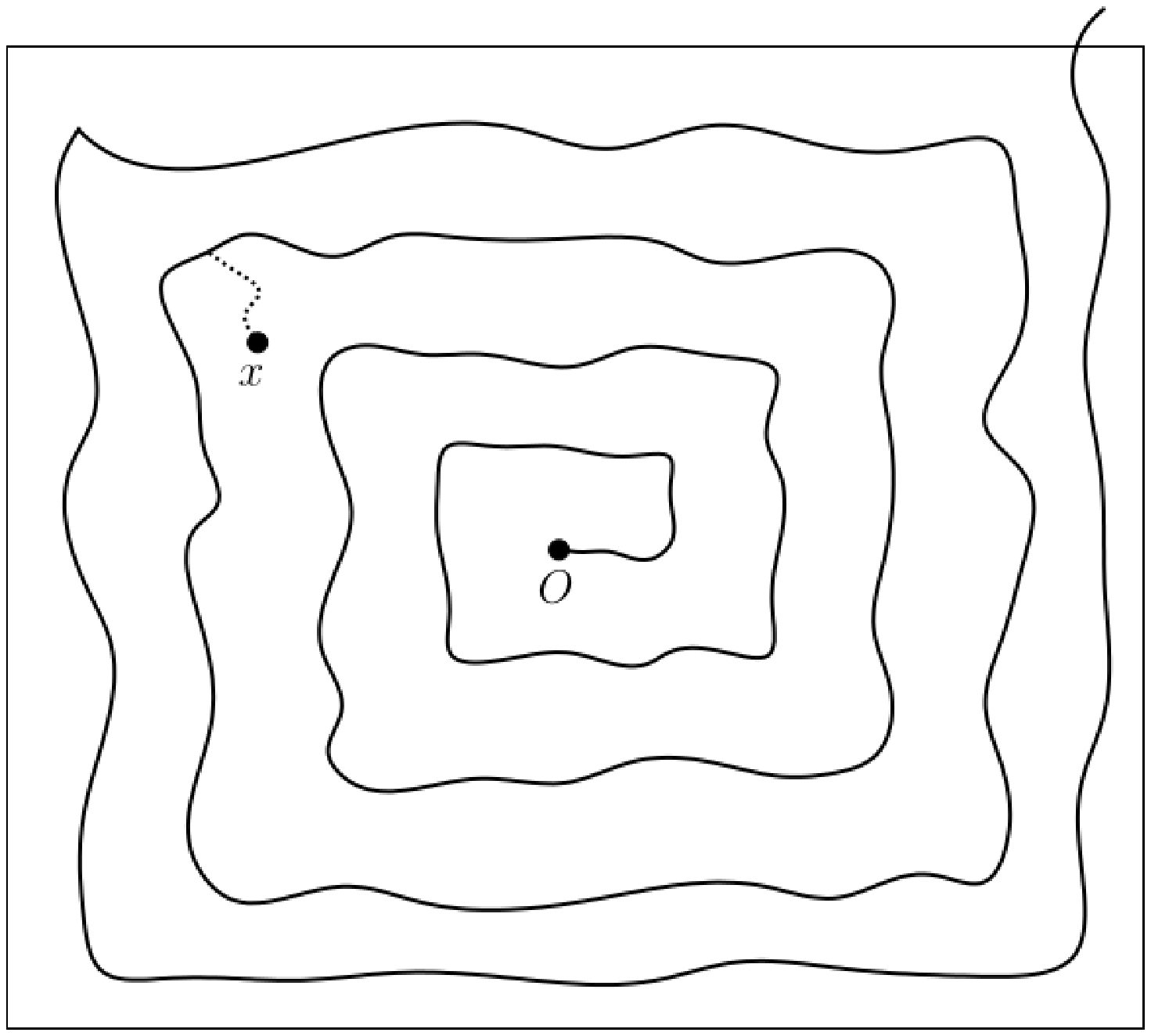}
\caption{Illustration for the spiral configuration. The solid curve stands for the unique infinite path in ${\cal U}$ started at the origin. We make it spiral around the origin many times. This configuration ensures that the intrinsic metric from the origin to $x$ is unusually large if the Euclidean metric between the origin and $x$ is not small.}\label{spiral}
\end{center}
\end{figure}

\begin{remark}\label{remarkd=3}
Rigorously speaking, it is not clarified that the uniform spanning tree on $\mathbb{Z}^{3}$ exhibits \textit{different} exponents than the high-dimensional case since the only information about the growth exponent $\beta$ is that it satisfies $1 < \beta \le 5/3$.  As shown in \eqref{spectral-dim}
, the leading order of the on-diagonal heat kernel is $n^{-\frac{3}{3 + \beta}}$ a.s. in three dimensions, while it equals $n^{-\frac{2}{3}}$ for every component of the uniform spanning forest in higher dimensions \cite{Hal-Hut, Hut-2}.
\end{remark}

We briefly discuss the strategies of our proofs here. It is harder to estimate volume and heat kernel fluctuations for uniform spanning trees than for Galton-Watson trees since uniform spanning trees do not have independence between disjoint subgraphs. In order to obtain Theorem \ref{main-thm-2}, we consider the three-dimensional uniform spanning tree as a collection of small pieces where the probability of events corresponding to those on the whole tree can be calculated. 
Similarly to \cite{BCK21}*{Theorem 1.1}, we consider unlikely configurations of  ${\cal U}$, 
namely ``comb'' and ``spiral'' configurations as depicted in Figures \ref{comb} and \ref{spiral} respectively. Here the comb configuration is constructed in such a way that we obtain an intrinsic ball in ${\cal U}$ with an unusually large size, which enables us to obtain (1.4). 
On the other hand, we use the spiral configuration to make an unusually small ball for the sake of the derivation of \eqref{main-2-1}. 
Although the scenario is 
essentially the same as that for the two-dimensional case in \cite{BCK21}, here we need to deal with a central  hurdle: since the Beurling projection theorem (a property of the simple random walk on $\mathbb{Z}^2$ that it hits any path of $\mathbb{Z}^2$ with high probability, see \cite{LaLi10}*{Theorem 6.8.1} for example) is not available when $d=3$, the construction of such unlikely configurations of ${\cal U}$ via Wilson's algorithm (see Section 2.2 for this algorithm) requires some extra work, which is rather complicated.  We overcome this difficulty through careful use of a 
type of hittability of loop-erased random walks in $\mathbb{Z}^{3}$, as derived in \cite{SaSh18}*{Theorem 3.1}. Combining Theorem \ref{main-thm-2} with the fact that the behavior of the effective resistance metric on ${\cal U}$ is similar to that of the intrinsic metric (see Section 4.2 below for this), Theorem \ref{main-thm} is also proved.

Before we end this section, let us explain the organization of this article. General notation together with background on the uniform spanning tree and loop-erased random walk in $\zzz$ will be introduced in Section 2. 
Then the claim 
(1.4) will be proved in Section 3 and \eqref{main-2-1} will be shown in Section 4. 
Finally, we will show Theorem \ref{main-thm} in Section 5.

\vspace{4mm}
\noindent {\bf Acknowledgements:}
The authors thank David A.~Croydon for his kind explanation of results and techniques in \cite{BCK21}, which enables us to prove the main theorems above. We would also like to thank him for his valuable comments on this article. DS is supported by a JSPS Grant-in-Aid for Early-Career Scientists, 18K13425 and JSPS KAKENHI Grant Number  17H02849, 18H01123, 21H00989, 22H01128 and  22K03336. 
SW is supported by JST, the Establishment of University Fellowships Towards the Creation of Science Technology Innovation, Grant Number JPMJFS2123.


\section{Definitions and background}



	We begin by introducing some notation for subsets of $\zzz$.
	Given two points $x,y\in\zzz$ and a set $A\subset\zzz$, 
		we let $d_E(x,y)=|x-y|$ be the Euclidean distance between $x$ and $y$ and let 
		$\mathrm{dist}(x,A)=\inf\{d_E(x,y) \mathrel{:} y\in A\}$. 
		
	For a set $A\subset\zzz$, we define the inner boundary $\partial_i A$ and the outer boundary
		$\partial A$ of $A$ as follows:
		\begin{align*}
			\partial_i A&=\{x\in A\mathrel{:} \mbox{there exists }y\in\zzz\setminus A\mbox{ such that }|x-y|=1\},	\\
			\partial A&=\{x\in\zzz\setminus A\mathrel{:} \mbox{there exists }y\in A\mbox{ such that }|x-y|=1\}.
		\end{align*}
			
	Given two points $x,y\in\zzz$, we write $x\sim y$ if $|x-y|=1$. 
	A finite or infinite sequence of vertices $\gamma=(\gamma_0,\gamma_1,\cdots)$ is called
		a \textbf{path} if $\gamma_{i-1}\sim \gamma_i$. 
	For two paths $\gamma=(\gamma_0,\gamma_1,\cdots, \gamma_k)$ and 
		$\gamma'=(\gamma'_0,\gamma'_1,\cdots)$ with $\gamma_k=\gamma'_0$, 
		we define the concatenation $\gamma\oplus \gamma'$ of them by 
	\[
		\gamma\oplus \gamma'=(\gamma_0,\gamma_1,\cdots,\gamma_k,\gamma'_1,\cdots).
	\]

	\subsection{Loop-erased random walk}
	Given a finite path $\gamma=(\gamma_0,\cdots,\gamma_k)$, 
		we let $\mathrm{len}(\gamma)=k$ be the length of $\gamma$ 
		and let $\mathrm{LE}(\gamma)$ be the chronological loop erasure of $\gamma$, 
		which is defined as follows.
	Set 
	\[
		T(0)=\sup\{j\mathrel{:} \gamma_j=\gamma_0\}
	\]
	and $\widetilde{\gamma}_0=\gamma_{T(0)}$. 
	Inductively, we set 
	\begin{equation}\label{defLE}
		T(i)=\sup\{j\mathrel{:}\gamma_j=\gamma_{T(i-1)+1}\},\ \ \ 
		\widetilde{\gamma}_i=\gamma_{T(i)}.
	\end{equation}
	Let
	\[
		l=\inf\{j\mathrel{:}T(j)=k\}.
	\]
	Then, $\mathrm{LE}(\gamma)=(\widetilde{\gamma}_0,\cdots,\widetilde{\gamma}_l)$. 
	The \textbf{loop-erased random walk (LERW)} is the random simple path obtained as the 
		loop-erasure of a path of the simple random walk. 

	The exact same definition also applies to the infinite simple random walk (SRW) $S$ on $\zzz$.
	Since $S$ is transient, the times $T(i)$ in (\ref{defLE}) are finite almost surely 
		for every $i\in\mathbb{Z}$. 
	The infinite simple path $\mathrm{LE}(S)$ is called the 
		\textbf{infinite loop-erased random walk (ILERW)}. 
	
	Now we introduce the growth exponent of the three-dimensional LERW. 
	We start SRW on $\zzz$ at the origin and run until it exits the ball of radius $n$ 
		centered at the origin. 
	Let $M_n$ be the length of its loop-erasure. 
	We denote the law of $S$ by $P$ and the corresponding expectation by $E$. 
	The growth exponent is defined by the limit
	\begin{equation}\label{growthbeta}
		\beta\coloneqq\lim_{n\to\infty}\frac{\log E(M_n)}{\log n},
	\end{equation}
	if exists. 	
	It is proved that the limit exists in \cite{S18} and that $\beta\in(1,5/3]$ in \cite{L99}. 
	Numerical estimates suggest that $\beta=1.624\cdots$, see \cite{W10}. 
	Moreover, following exponential tail bound of $M_n$ is shown in \cite{S18}. 
	\begin{theo}(\cite{S18}*{Theorem 1.1.4})
	There exists $c>0$ such that for all $n\ge 1$ and $\kappa\ge 1$,
	\[
		\prob(M_n\ge \kappa E(M_n))\le 2\exp\{-c\kappa\},
	\]
	and for any $\vep\in(0,1)$, there exist $0<c_\vep,C_\vep<\infty$ such that for 
		all $n\ge 1$ and $\kappa\ge 1$, 
	\[
		P(M_n\le \kappa^{-1}E(M_n))\le C_\vep\exp\{-c_\vep\kappa^{\frac{1}{\beta}-\vep} \}.
	\]
	\end{theo} 
	
	\subsection{Uniform spanning tree}
	A subgraph of a connected graph $G$ is called a \textbf{spanning tree} on $G$ 
		if it is connected, contains all vertices of $G$ and has no cycle. 
	Let $\mathcal{T}(G)$ be the set of all spanning trees on $G$. 
	For a finite connected graph $G$, a random tree choosen according to 
		the uniform measure on $\mathcal{T}(G)$ is called 
		the \textbf{uniform spanning tree (UST)} on $G$. 
	We can define the uniform spanning tree on $\zzz$, or the three-dimensional uniform spanning tree, 
		as the weak limit of the USTs on the finite boxes $\zzz\cap [-n,n]^3$, 
		see \cite{P91}.
	
	We will assume that the three-dimensional UST $\ust$ is built on 
		a probability space $(\Omega,\mathcal{F},\prob)$ and 
		we denote the corresponding expectation by $\mathbf{E}$. 
	Note that, $\prob$-a.s., $\ust$ is a one-ended tree (\cite{P91}). 
	For any $x,y\in\zzz$ and any connected subset $A\subset\zzz$, 
		we write $\gamma(x,y)$ for the unique self-avoiding path between $x$ and $y$, 
		$\gamma(x,A)$ for the shortest path among $\{\gamma(x,y)\mathrel{:}y\in A\}$ 
		if $x\not\in A$, and $\gamma(x,A)=\{x\}$ if $x\in A$. 
	We let $\gamma(x,\infty)$ for the unique infinite self-avoiding path started at $x$. 
	We denote by $d_\ust$ the intrinsic metric on the graph $\ust$, 
		\ie $d_\ust(x,y)=\mathrm{len}(\gamma(x,y))$. 
	Similarly, we denote by $d_E$ the Euclidean metric on $\zzz$ and 
		for any $x\in\zzz$ and connected $A\subset \zzz$, 
		we let $d_E(x,A)=\inf_{y\in A} d_E(x,y)$
	
	We define balls in the intrinsic metric by
	\begin{equation}\label{intrball}
		B_\ust(x,r)=\{y\in\zzz \mathrel{:} d_\ust(x,y)\le r\}
	\end{equation}
	and let $|B_\ust(x,r)|$ be the number of points in $B_\ust(x,r)$. 
	We denote balls in the Euclidean metric by
	\begin{equation}
		B(x,r)=\{y\in\zzz \mathrel{:} d_E(x,y)\le r \},
	\end{equation}
	and balls in $l_\infty$-metric $d_\infty$, \ie cubes, by
	\begin{equation}
		B_\infty(x,r)=\{y\in\zzz \mathrel{:} d_\infty(x,y)\le r\}.
	\end{equation}

	Throughout the paper, we let $S^z$ denote a simple random walk on $\zzz$ started at $z\in\zzz$ 
		and $P^z$ denote its law.
	We take $(S^z)_{z\in\zzz}$ to be independent. 
	
	Now we recall Wilson's algorithm. 
	This method to construct UST with LERW was first introduced to finite graphs (\cite{W96})
		and then extended to transient $\zd$ including $\zzz$ (\cite{BLPS01}). 
	Let $\{v_1,v_2,\cdots\}$ be an ordering of the vertices of $\zzz$ and 
		let $\gamfty$ be the infinite LERW started at the origin. 
	Given a path $\gamma$ and a set $A\subset\zzz$, 
		we let $\tau(A)=\tau_\gamma(A)=\min\{i\ge 0\mathrel{:}\gamma_i\in A\}$. 
	We define a sequence of subtrees of $\zzz$ inductively as follows: 
	\begin{align*}
		\ust_0&=\gamfty,	\\
		\ust_i&=\ust_{i-1}\cup\mathrm{LE}(S^{z_i}[0,\tau(\ust_{i-1})]),\ i\ge 1,	\\
		\ust&=\cup_i \ust_i.
	\end{align*}
	Then by \cite{BLPS01}, the random tree $\ust$ has the same law 
		as the three-dimensional UST. 
	It follows that the law of $\ust$ above does not depend on the ordering of $\zzz$.

	Finally we define the simple random walk on $\ust$. 
	We denote by $\mu_\ust$ the measure on $\zzz$ such that $\mu_\ust(\{x\})$ is given 
		by the number of edges of $\ust$ which contain $x$.  
	For a given realization of $\ust$, the simple random walk on $\ust$ is the discrete time 
		Markov process $X^\ust=((X_n^\ust)_{n\ge 0},(P_x^\ust)_{x\in\zzz})$ 
		which at each step jumps from its current location to a uniformly chosen neighbor in $\ust$. 
	For $x\in \zzz$, the law $P_x^\ust$ is called the \textbf{quenched law} of the simple random 
		walk on $\ust$. 
	We write 
	\begin{equation}\label{hkdef}
		p_n^\ust (x,y)=\frac{P_x^\ust(X_n^\ust=y)}{\mu_\ust(\{y\})},\ x,y\in\zzz,
	\end{equation}
	for the \textbf{quenched heat kernel}. 

	\subsection{Effective resistance}
	Now we introduce the effective resistance between two subgraphs of a connected subgraph, 
		which we will use to estimate upper heat kernel fluctuaions in Section 4. 
	\begin{defi}
	Let $G=(V,E)$ be a connected graph. 
	For functions $f$ and $g$ on $V$, we define a quadratic form $\mathcal{E}$ by 
	\[
		\mathcal{E}(f,g)=\frac{1}{2}\sum_{\substack{x,y\in V \\ x\sim y}} (f(x)-f(y))(g(x)-x(y)).
	\]
	If we regard $G$ as an electrical network with a unit resistance on each edge in $E$, 
		then for disjoint subsets $A$ and $B$ of $V$, the \textbf{effective resistance} between $A$ and $B$ is 
		defined by 
	\begin{equation}\label{defofreff}
		\Reff{A}{B}^{-1}=\inf\{\mathcal{E}(f,f)\mathrel{:}\mathcal{E}(f,f)<\infty, f|_A=1,f|_B=0\}.
	\end{equation}
	Let $\Reff{x}{y}=\Reff{\{x\}}{\{y\}}$.
	\end{defi}
	It is known that $\Reff{\cdot}{\cdot}$ is a metric on $G$, see \cite{Wei19}. 

\section{Upper volume fluctuations}


	In this section, we prove (\label{main-2-2}), upper volume fluctuations of 
		log-logarithmic magnitude in Theorem \ref{thmvup}. 
	The key ingredient of the proof is the following lemma, 
		which provides a lower bound on an upper tail of the volume of intrinsic balls in the three-dimensional UST $\ust$. 
	
	Recall that $\beta$ is the growth exponent of the three-dimensional LERW defined by 
		(\ref{growthbeta}) and $B_\ust$ stands for intrinsic balls in $\ust$ defined by (\ref{intrball}). 

	\begin{prop}\label{VLower}
		Let $\ust$ be the three-dimensional UST build on a probability space $(\Omega,\mathcal{F},\prob)$. 
		Then there exist $c_1,c_2>0$ such that 
			for all $\lambda>0$ and $r\ge 1$,
		\begin{equation}
			\prob(|B_\ust(0,r)|\ge \lambda r^{3/\beta})\ge c_1\exp\{-c_2\lambda^{(\beta-1)/\beta}\log \lambda \} \label{vlower},
		\end{equation}
		where $\beta$ is the growth exponent of the three-dimensional LERW. 
	\end{prop}
	
	\begin{remark}
		See \cite{ACHS20}*{Proposition 6.1} 
			for an exponential upper bound for the probability in (\ref{vlower}).
	\end{remark}

\subsection{The comb configuration in the UST}	
	We explain an idea of the proof of (\ref{vlower}) here, which is inspired by 
		the proof of (4.1) of \cite{BCK21}*{Lemma 4.1}. 
	We consider a cube of $N\times N\times N$ boxes, each of size $m\times m\times m$. 
	WE take a box on a corner of the cube and let the box be centered at the origin. 
	Described in Figure 1 is the boxes which has an intersection with the 
		$xy$ plain.  
	For each $j\ge 0$, let 
	\begin{equation}\label{xjdefi}
		x_j=(jm,0,0)\in\zzz, \ \ \ \Bx{j}=B_\infty(x_j,m/2), 
	\end{equation}
	\ie $\Bx{j}$ is the cube of side-length $m$ centered at $x_j$ (see Section 2.2 for the definition of $B_\infty$).
	
	Let $\gamfty$ be the infinite LERW started at the origin, which is the first branch in Wilson's algorithm to generate 
		$\ust$. 
	We consider the event $A_{x_1}$ which is the intersection of the following events:
	\vspace{-0.5\baselineskip}
	\begin{itemize}
		\item $\gamfty$ moves toward the right until it exits from a ``tube'' $\bigcup_{j=1}^N \Bx{j}$ without backtracking.
		\item The number of points in $\gamfty\cap \Bx{j}$ is bounded above by $m^\beta$ for all $1\le j\le N$.
		\item For some snall $\vep>0$, with high probability, each point in $B(x_j,\vep m)~(j=1,\cdots N)$ is connected to $\gamfty$ with a path of length of order $m^\beta$. 
	\end{itemize}
	\vspace{-0.5\baselineskip}
	
	Next we run SRWs $S^{(j)},~j=1,2,\cdots,N$ independent of $\gamma$ and each other 
		started at the points $(jm,Nm,0)$. 
	We consider the event $A_{x_2}$ where each $S^{(j)}$ moves in a ``tube'' parallel to the $y$ axis until it hits $\gamma$, 
		the number of points in its loop erasure is bounded above by $Nm^\beta$ 
		and every point in a small Euclidean ball around the center of each box is connected to the loop erasure with a short path. 
		
	Finally, we consider the corresponding event $A_{x_3}$ for independent SRWs started at $(jm,km,Nm)$ until they hit 
		the already constructed subtree in the tubes parallel to the $z$ axis. 

	Note that if the intersection $A_{x_1}\cap A_{x_2}\cap A_{x_3}$ occurs, it leads to a lower bound of the volume of a ball in intrinsic metric. 
	Once we have a lower bound of the probability of the event $A$, we consider LERWs 
		satisfying the same condition and paralell to $y$ and $z$ axis. 

	\begin{figure}[htb]
		\centering
		\includegraphics[width=0.6\linewidth]{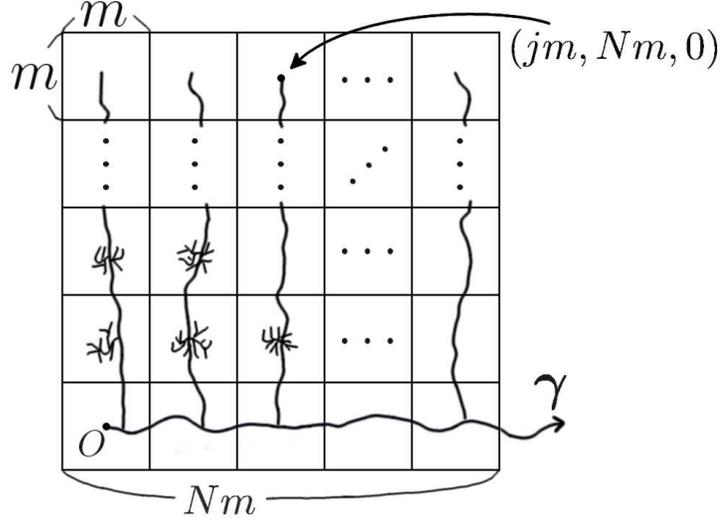}\label{volupperfig}
		\caption{The event $A_{x_1}\cap A_{x_2}\cap A_{x_3}$ to consider for upper volume fluctuations}
	\end{figure}
	
	In the remainder of this subsection, we will establish a lower bound of the probability of the 
		event $A_{x_1}\cap A_{x_2}\cap A_{x_3}$ in Lemma \ref{balls} and in Lemma \ref{x1axis}. 
	In order to do so, we follow an argument of \cite{LS21}, Section 4, which make use of the cut 
		points of the three-dimensional SRW. 

	We begin with defining several events. 
	\begin{defi}\label{defsec3}
	For $a<b$, we define 
	\begin{align*}
		Q[a,b]&=\{(x^1,x^2,x^3)\in\zzz\mathrel{:} a\le x^1\le b, -m\le x^2,x^3\le m\},	\\
		Q(a)&=\{(x^1,x^2,x^3)\in\zzz\mathrel{:}x^1=a, -m\le x^2,x^3\le m\}.
	\end{align*}
	We also set 
	\begin{align}
		\tQ(a)=\{(x^1,x^2,x^3)\in\zzz\mathrel{:}x^1&=a, -m/2\le x^2,x^3\le m/2\}, \notag	\\
		R_j=\{(x^1,x^2,x^3)\in\zzz\mathrel{:}x^1&=jm,\ |x^2|^2+|x^3|^2< m^2/100\} \notag	\\
		\cup&\{(x^1,x^2,x^3)\in\zzz \mathrel{:} x^1=jm,\ |x^2|^2+|x^3|^2> m^2/64\}. \label{R_j}
	\end{align}
	\end{defi}
	
	\begin{figure}[htb]
		\centering
		\includegraphics[width=0.6\linewidth]{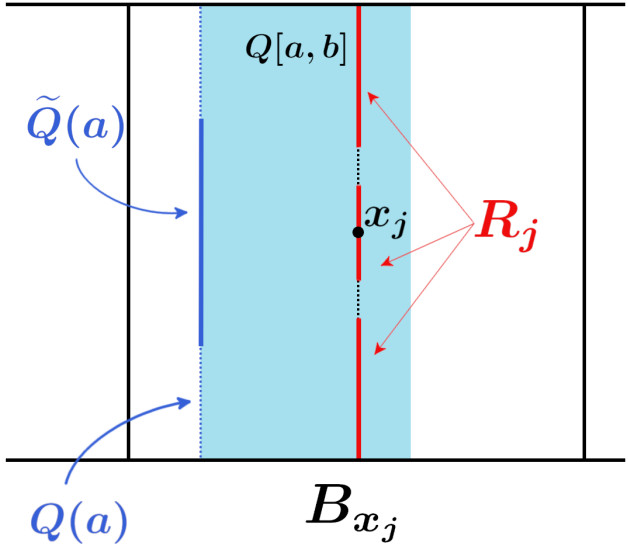}\label{figureQR}
		\caption{Sets $Q[a,b]$, $Q(a)$, $\widetilde{Q}(a)$ and $Rj$}
	\end{figure}
	
	Note that setting $a_j=(j-1/2)m$, it follows that $Q[a_j,a_{j+1}]=\Bx{j}$ and that 
		$Q(a_{j+1})$ corresponds to the right face of $\Bx{j}$ (see (\ref{xjdefi}) for the definition of $\Bx{j}$) 

	Now we consider the SRW $S$ on $\zzz$ started at the origin. 
	By linear interpolation, we may assume that $S(k)$ is defined for every non-negative 
		real $k$ and $S[0,\infty)$ is a continuous curve. 
	For a continuous curve $\lambda$ in $\rrr$ , we define 
	\[
		t_\lambda(a)=\inf\{k\ge 0\mathrel{:} \lambda(k)\in Q(a)\}.
	\]
	
	Let 
	\begin{align}
		N&=(\log\log m)^{1/2},	\\
		q&=m/N^2. \label{defofq}
	\end{align}
	Using $t_S(a)$, we define events $A_j$ as following: 
	\begin{align}
		A_0=\{\tS{1}&<\infty, S(\tS{1})\in \tQ(a_1), S[0,\tS{1}]\subset \Bx{0}, S[t_S(a_1-q),\tS{1}]\cap Q(a_1-2q)=\emptyset\},	\notag	\\
		A_j=\{\tS{j}&<\tS{j+1}<\infty, S(\tS{j+1})\in\tQ(a_{j+1}), S[\tS{j},\tS{j+1}]\subset Q[a_j-q,a_{j+1}]\setminus R_j,	\notag	\\
		\label{A_j}S&[t_S(a_{j+1}-q),\tS{j+1}]\subset Q[a_{j+1}-2q,a_{j+1}]\}\ \mbox{for}\ j\ge 1.
	\end{align}
	The event $A_0$ guarantees that $S$ exits $\Bx{0}$ from $\tQ(a_1)$ and has no big backtracking 
		from $t_S(a_1-q)$ to $\tS{1}$. 
	For $j\ge 1$, the event $A_j$ ensures that once $S$ enters $\Bx{j}$, 
		it keeps moving to the right until hitting $Q(a_{j+1})$. 
	The last condition of $A_j$ requires that 
		$S$ has no big backtracking in $[t_S(a_{j+1}-q),\tS{j+1}]$. 
	We note that the event $A_0$ (resp. $A_j,~j\ge 1$) is measurable with respect to 
		$S[0,\tS{1}]$ (resp. $S[\tS{j},\tS{j+1}]$). 
		
	\begin{figure}[htb]
		\centering
		\includegraphics[width=0.6\linewidth]{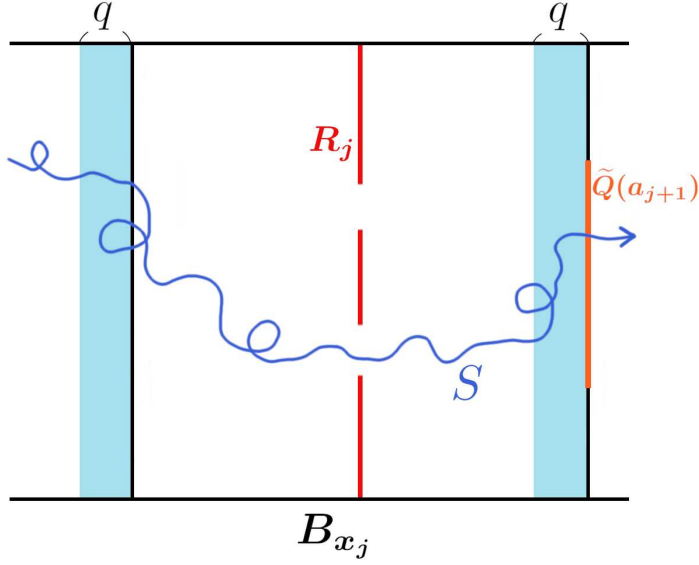}\label{figureAj}
		\caption{Definition of the event $A_j\ (j\ge 1)$}
	\end{figure}
		
	We set 
	\begin{equation}\label{G_j}
		G_j=\bigcap_{k=0}^j A_k.
	\end{equation}

	We next consider a cut time with special properties for the SRW. 
	
	\begin{defi}\label{cuttimedef}
	Suppose that the event $A_j$ defined in (\ref{A_j}) occurs. 
	For each $j\ge 1$, we call $k$ is a nice cut time in $\Bx{j}$ if it satisfies the following conditions: 
	\begin{enumerate}[(i)]
		\item $t_S(a_j+\frac{q}{2})\le k\le t_S(a_j+q)$, 
		\item $S[\tS{j},k]\cap S[k+1,\tS{j+1}]=\emptyset$, 
		\item $S[k,\tS{j+1}]\cap Q(a_j)=\emptyset$, 
		\item $S(k)\in Q[a_j+\frac{q}{2},a_j+q]$. 
	\end{enumerate}
	If $k$ is a nice cut time in $\Bx{j}$, then we call $S(k)$ a nice cut point in $\Bx{j}$. 
	\end{defi}
	
	\begin{figure}[htb]
		\centering
		\includegraphics[width=0.6\linewidth]{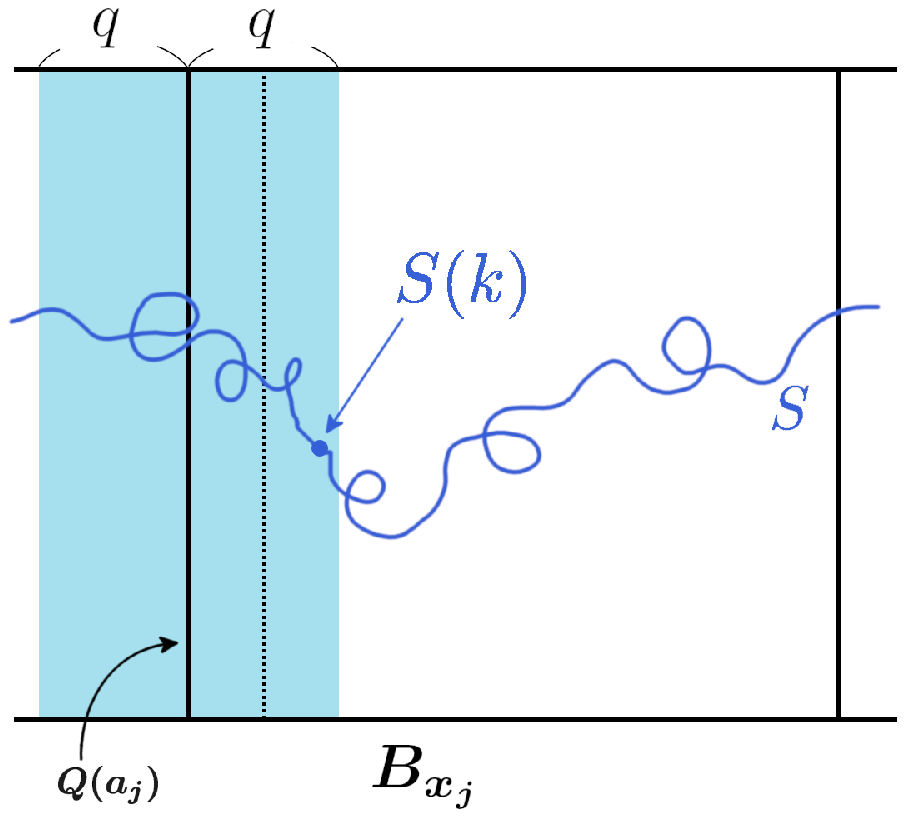}\label{figureQw}
		\caption{Example of nice cut point in $\Bx{j}$}
	\end{figure}
	
	We define events $B_j$ by 
	\[
		B_j=\{S \mbox{ has a nice cut point in }\Bx{j}\},
	\]
	for each $j\ge 1$. 
	Note that event $B_j$ is measurable with respect to $S[\tS{j},\tS{j+1}]$. 
	We define 
	\begin{equation}\label{H_j}
		H_j=\bigcap_{k=1}^j B_k.
	\end{equation}

	Now we consider two random curves $\xi_j$ and $\xi_j'$ defined as follows. 
	We set 
	\[
		\xi_j=\mathrm{LE}(S[0,\tS{j+1}]),
	\]
	and 
	\begin{align}
		\lambda_j&=\mathrm{LE}(S[\tS{j},\tS{j+1}])\mbox{\hspace{1em}for }j\ge 0, \label{lamdef}	\\
		\xi_0'&=\xi_0,\ \ \ \ \xi_j'=\xi_0'\oplus \lambda_1\oplus \cdots \oplus \lambda_j\mbox{\hspace{1em}for }j\ge 1.	\label{xidef}
	\end{align}

	\begin{figure}[htbp]
	\begin{minipage}[b]{0.31\linewidth}
		\centering
		\includegraphics[width=1\linewidth]{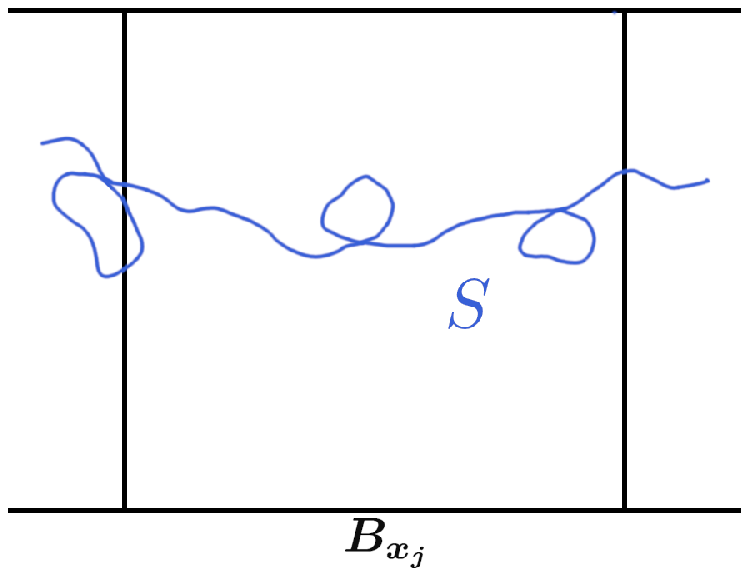}
		\subcaption{Original random walk $S$}
	\end{minipage}
	\hspace{0.01\linewidth}
	\begin{minipage}[b]{0.31\linewidth}
		\centering
		\includegraphics[width=1\linewidth]{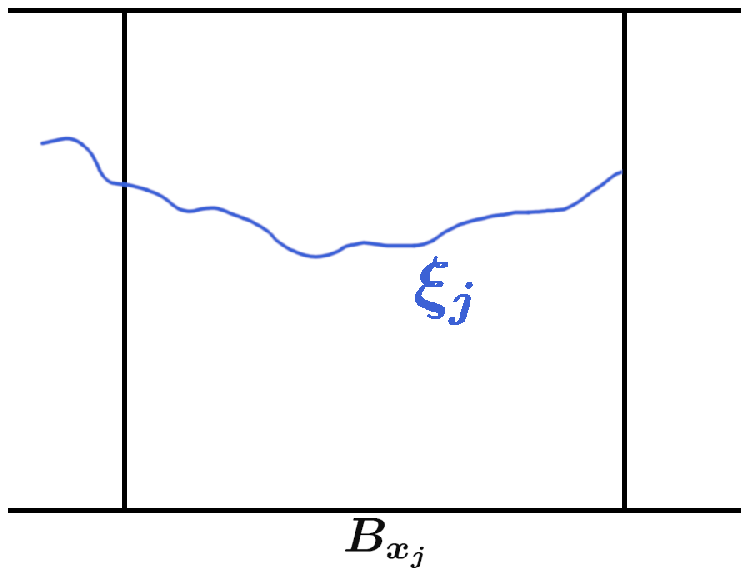}
		\subcaption{$\xi_j$}
	\end{minipage}
	\hspace{0.01\linewidth}
	\begin{minipage}[b]{0.31\linewidth}
		\centering
		\includegraphics[width=1\linewidth]{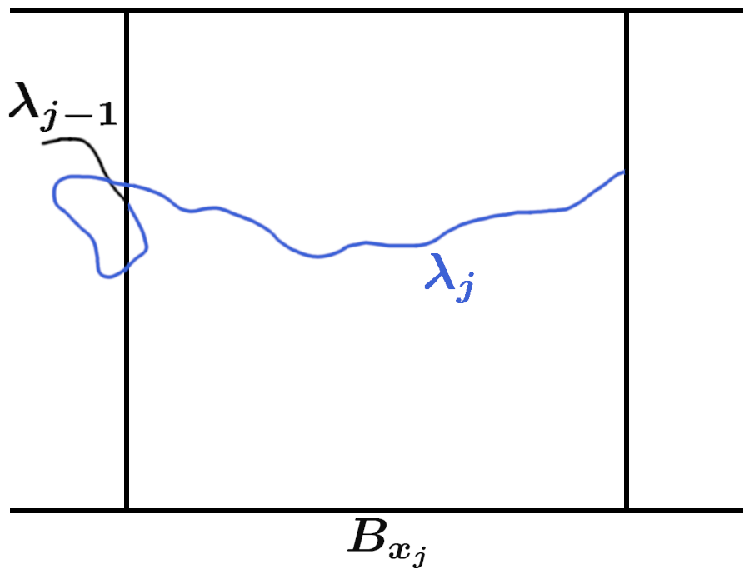}
		\subcaption{$\xi'_j$}
	\end{minipage}
	\caption{Examples of $\xi_j$ and $\xi'_j$}
	\end{figure}
	
	Note that $\xi_j'$ is not necessarily a simple curve and thus $\xi_j\neq\xi_j'$ in general. 
	However, the next lemma from \cite{LS21} shows that the difference between these two curves 
		is small on the event $G_j\cap H_j$. 

	\begin{lem}[\cite{LS21}*{Lemma 4.3}]
		Let $j\ge 1$.
		Suppose that $G_j\cap H_j$ defined in (\ref{G_j}) and (\ref{H_j}) occurs. 
		Then, for the length of $\xi_j$ and $\xi_j'$, we have
		\begin{equation}\label{2len}
			\mathrm{len}(\xi_j)\le \mathrm{len}(\xi_0')+\sum_{k=1}^j\left\{\mathrm{len}(\lambda_k)+\left|\xi_j\cap Q[a_k-q,a_k+q]\right|\right\},
		\end{equation}
		where for $A\subset\rrr$, we write $|A|$ for the number of points in $A\cap\zzz$. 
	\end{lem}
	Note that $\mathrm{len}(\xi_j')=\mathrm{len}(\xi_0)+\sum_{k=1}^j \mathrm{len}(\lambda_k)$, 
		and thus the above lemma compares the length of $\xi_j$ and $\xi_j'$. 
	
	We will next deal with the length and the hittability of each $\lambda_j$. 
	For $C\ge 1$, we define the event $E_j(C)$ by
	\begin{equation}
	\label{defE}
		E_0=E_0(C)=\{\mathrm{len}(\xi_0)\le Cm^\beta\},\ \ \ E_j=E_j(C)=\{\mathrm{len}(\lambda_j)\le Cm^\beta\}\ \ \mbox{for }j\ge 1,
	\end{equation}
	where $\xi_0$ and $\lambda_j,\ j\ge 1$ are as defined in (\ref{xidef}). 
	Let $R^z$ be a SRW on $\zzz$ started at $z\in\zzz$ and independent of $S$.  
	We denote by $P$ and $P^z$ the law of $S$ and $R^z$, respectively. 
	For $N\ge 4$ and $\eta>0$, we define the event $F_j(\eta)$ by
	\begin{equation}\label{defF}
		F_j=F_j(\eta)=\{P^{x_j}(\lambda_j\cap R^{x_j}[0,T_{R^{x_j}}(2m/5)]\neq\emptyset) \ge \eta\},
	\end{equation}
	where $T_R(r)=\inf\{k\ge 0\mathrel{:}|R(k)|\ge r\}$. 
	Note that $F_j(\eta)$ is measurable with respect to $S[\tS{j},\tS{j+1}]$. 
	\begin{figure}[htb]
		\centering
		\includegraphics[width=0.5\linewidth]{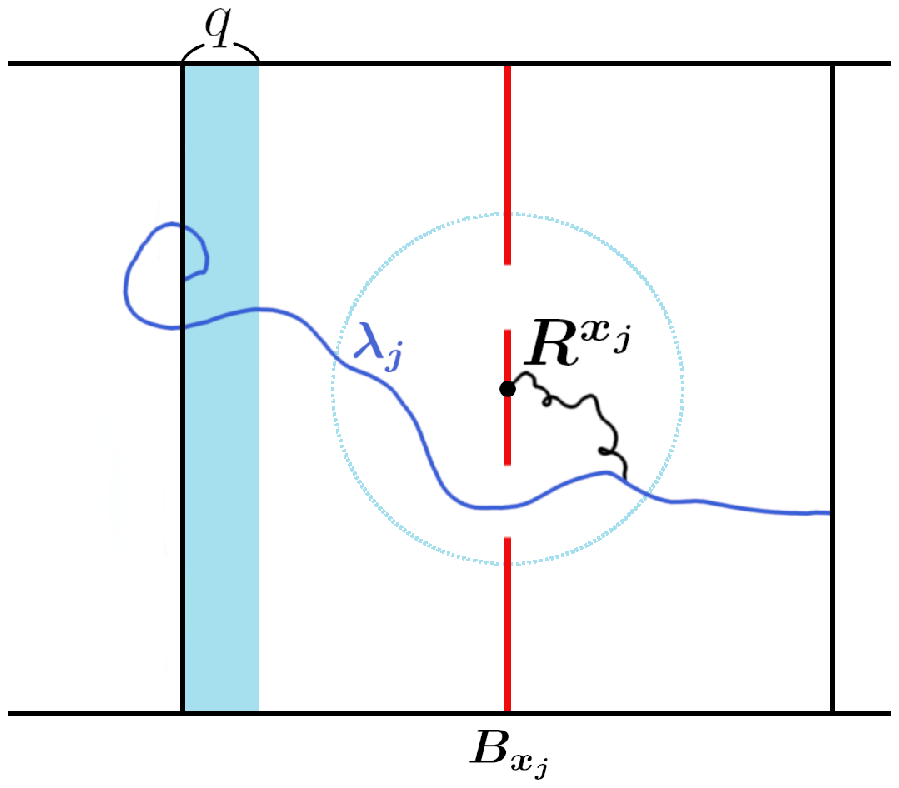}\label{figureFj}
		\caption{The event $\{\lambda_j\cap R^{x_j}[0,T_{R^{x_j}}(2m/5)]\neq\emptyset\}$}
	\end{figure}
	
	The next lemma gives a lower bound on the probability of 
		$A_j\cap B_j\cap E_j(C)\cap F_j(\eta)$ choosing $C$ sufficiently large 
		and $\eta$ sufficiently small. 
	\begin{lem}\label{c-eta}
	There exist universal constants $0<\eta_*,c_*,C_*<\infty$ such that 
	\begin{equation*}
			P(A_0\cap E_0(C_*))\ge c_*,
	\end{equation*}
	and for all $j\ge 1$,
	\begin{equation}\label{onebox}		
		\min_{x\in\tQ(a_j)}P^x(A_j\cap B_j\cap E_j(C_*)\cap F_j(\eta_*))\ge c_*N^{-2}.
	\end{equation}
	\end{lem}
	\begin{proof}
	The first assertion is proved in \cite{LS21}*{Lemma 4.4} 
		and we also follow its proof to show that (\ref{onebox}) holds. 
	By the translation invariance, the minimum in the left hand side of (\ref{onebox}) does not 
		depend on $j$. 
	Hence, we will only consider the case $j=1$. 
	
	It follows from the gambler's ruin estimate (\cite{LaLi10}*{Proposition 5.1.6}, for example) that 
	\begin{equation}\label{ABEF1}
		c_1N^{-2}\le P^x(A_1)\le c_2N^{-2}\ \ \ \mbox{uniformly in }x\in \tQ(a_1),
	\end{equation}
	and from \cite{CutTimes}*{Corollary 5.2} that 
	\begin{equation}\label{ABEF2}
		P^x(B_1\mid A_1)\ge c_3\ \ \ \mbox{uniformly in }x\in \tQ(a_1),
	\end{equation}
	for some universal constants $0<c_1,c_2,c_3<\infty$. 

	On the event $A_1\cap B_1$, 
		let $k_1$ be a nice cut time in $\Bx{1}$ as defined in Definition $\ref{cuttimedef}$. 
	By definition, it follows that $k_1\le t_S(a_1+q)$ and 
	\begin{equation}\label{36prf}
		\mathrm{dist}\left(x_1,\mathrm{LE}(S[k_1,\tS{2}]\right) \le m/8.
	\end{equation}
	We check this by contradiction. Suppose that (\ref{36prf}) does not hold. 
	This implies that $\mathrm{LE}(S[k_1,\tS{2}])$ contains some point $z\in R_1$, 
		where $R_1$ is as defined in $(\ref{R_j})$. 
	Thus, it also holds that $z\in S[\tS{1},\tS{2}]$, which contradicts (\ref{A_j}). 
	
	It follows from (\ref{36prf}) and the decomposition 
		$\lambda_1=\mathrm{LE}(S[\tS{1},k_1])\oplus \mathrm{LE}(S[k_1,\tS{2}])$ 
	that $\mathrm{dist}(x_1,\lambda_1)\le m/8$.
	Hence, by \cite{SaSh18}*{Theorem 3.1}, 
		there exist some universal constant $0<\eta_1,c_4<1$ such that 
	\[
		P^x\left(F_1(\eta_1) \mid A_1\cap B_1\right) \ge 1-c_4\ \ \ \mbox{uniformly in }x\in \tQ(a_1).
	\]
	Combining this with (\ref{ABEF2}), we obtain 
	\begin{equation}\label{ABEF3}
		P^x(B_1\cap F_1(\eta_1)^c\mid A_1)=P^x(F_1(\eta_1)^c\mid A_1\cap B_1)P^x(B_1\mid A_1)\le c_3c_4.
	\end{equation}
	
	By the similar argument to the proof of \cite{LS21}*{Lemma 4.4}, 
		we can obtain $E^x(\mathrm{len}(\lambda_1))\le Cm^\beta$ uniformly in 
		$x\in \widetilde{Q}(a_1)$ and by the Markov's inequality, 
		there exists a universal constant $0<C_1<\infty$ such that 
	\[
		P^x(E_1(C_1)^c\mid A_1)\le c_3(1-c_4)/2,
	\]
	uniformly in $x\in\tQ(a_1)$. 
	Combining this with (\ref{ABEF1}), (\ref{ABEF2}) and (\ref{ABEF3}) yields
	\[
		P^x(A_1\cap B_1\cap E_1(C_1)\cap F_1(\eta_1))\ge \frac{c_1c_3(1-c_4)}{2}N^{-2},
	\]
	which finishes the proof. 
	\end{proof}
	


	Now we perform Wilson's algorithm around the center of each small cube. 
	Recall that $B(x,r)$ indicates the ball in the Euclidean metric 
		and $\tau_\gamma(A)$ is the first time that $\gamma$ hits $A$. 
	Given $\lambda'_i$ (see (\ref{lamdef}) for the definition), we regard it as a deterministic set and 
		consider independent simple random walks started at the points in $B(x_i, \lambda^{-2}m)$ 
		for some $\lambda\ge 1$. 
	We regard these random walks as a step of Wilson's algorithm rooted at $\xi'_i$.  

	In the following lemma, we will observe that with high conditional probability, 
		a small Euclidean ball around the center of each box $\Bx{i}$ is included in an intrinsic ball centered 
		at the same point and of radius of order $m^\beta$. 
	We define the event $M_j(\lambda)$ by 
	\begin{equation}
		M_j(\lambda)=\{B(x_j,\lambda^{-2}m)\subset B_{\ust^N}(x_j,\lambda^{-1}m^\beta\}.
	\end{equation}
	
	\begin{lem}\label{balls}
	There exist $c_4,c_5>0$ such that for all $\delta>0,m\ge 1,\lambda\in[1,m^{(1-\delta)/2})$ and $j\in\{1,\cdots,N\}$, 
	\begin{equation}\label{smball}
		\prob\left(M_j(\lambda)) \mid A_j\cap B_j\cap E_j(C_*)\cap F_j(\eta_*) \right) \ge 1-c_4\lambda^{-c_5}.
	\end{equation}
	\end{lem}

	\begin{proof}
	It suffices to show (\ref{smball}) in the case $j=1$. 
	Let $\prob_{\ust^N}(\cdot)\coloneqq\prob(\cdot\mid A_1\cap B_1\cap E_1(C_*)\cap F_1(\eta_*))$. 
	We may assume that $m$ and $\lambda$ are sufficiently large for the same reason 
		as \cite{ACHS20}*{Proposition 4.1}. 
	Thus, we take large $m$ so that 
	\begin{equation}\label{mislarge}
		\frac{m^{\delta/2}}{\delta\log m+2}\ge 10,
	\end{equation}
	for a fixed $\delta>0$. 
	
	Recall that $R^z$ indicates the SRW on $\zzz$ started at $z$ and independent of $S$. 
	Given $S$, we run $R^{x_1}$ until it hits $\xi_N$. 
	On the event $A_1\cap B_1\cap E_1(C_*)\cap F_1(\eta_*)$, 
		we have that $d_E(x_1,\xi_N)\in [m/10,m/8]$ by the definition of $A_1$ (see (\ref{A_j})) and 
		the event $\left\{\mathrm{LE}(R^{x_1}[0,\tau_{R^{x_1}}(\xi_N)])\subset\Bx{1}\right\}$ occurs with positive conditional probability by the definition of $F_1$ and $\eta_*$ (see (\ref{defF}) and (\ref{onebox})). 
	By \cite{M09}*{Corollary 4.5}, we have that for $\lambda\ge 40$ the law of $\mathrm{LE}(R^{x_1}[0,\tau_{R^{x_1}}(\xi_N)])$ restricted to $B(x_1,\lambda^{-1}m)$ is 
		comparable to that of the infinite LERW started at $x_1$ restricted to the same ball. 
	Thus, we can follow the discussion of \cite{ACHS20}*{Proposition 4.1}, 
		which gives a tail bound estimate of the volume of intrinsic balls in the three-dimensional UST. 

	Let $\sigma$ and $\widetilde{\sigma}$ be the first time that 
		$\gamma_{x_1}\coloneqq \mathrm{LE}(R^{x_1}[0,\tau_{R^{x_1}}(\xi_N)])$ exits 
		$B(x_i,\lambda^{-2}m)$ and $\ B(x_i,\lambda^{-1}m)$, respectively. 
	We define the event $F$ by 
	\[
		F=\left\{\gamma_{x_1}[\widetilde{\sigma},\mathrm{len}(\gamma_{x_1})]\cap B(x_1,2\lambda^{-2}m)=\emptyset,\ \sigma\le \frac{1}{2}\lambda^{-1}m^\beta\right\}.
	\]
	Then by \cite{L96}*{Proposition 1.5.10}, the probability that a SRW started at a point 
		outside $\ B(x_1,\lambda^{-1}m)$ returns to $B(x_1,\lambda^{-2}m)$ is 
		smaller than $C\lambda^{-1}$ for some universal constant $C<\infty$. 
	This implies $\prob_{\ust^N}(\gamma_{x_1}[\widetilde{\sigma},\mathrm{len}(\gamma_{x_1})]\cap B(x_1,2\lambda^{-2}m)\neq\emptyset)\le C\lambda^{-1}$. 
	On the other hand, by \cite{S18}*{Theorem 1.4} and \cite{LS19}*{Corollary 1.3}, 
		we have that the probability that 
		$\sigma$ is greater than $\frac{1}{2}\lambda^{-1}m^\beta$ is bounded above 
		by $C\exp\{-c\lambda^{-1}\}$ for some universal constants $0<C,c<\infty$. 
	Thus, it follows from the above estimates that 
	\begin{equation}\label{ball_F}
		\prob_{\ust^N}(F)\ge 1-C\lambda^{-1}.
	\end{equation}
	
	Next we observe that $\gamma_{x_1}$ can be hit by another independent SRW 
		started at a point which is close to $\gamma_{x_1}$ with high probability. 
	For $\zeta>0$, we define an event $G(\zeta)$ by 
		\[
			G(\zeta)=\left\{\forall y\in B(x_1,2\lambda^{-2}m),\ P_R^y(R[0,T_{R^y}(x_1,\lambda^{-3/2}m)]\cap \gamma_{x_1}=\emptyset)\le \lambda^{-\zeta}\right\},
		\]
		where $T_{R^y}(x,l)$ is the first time that $R^y$ exists $B(x,l)$. 
	From \cite{SaSh18}*{Theorem 3.1}, there exist universal constants $C<\infty$ 
		and $\zeta_1\in(0,1)$ such that for all $m\ge 1$ and $\lambda\ge 2$, 
	\begin{equation}\label{zeta_1}
		\prob(G(\zeta_1))\ge 1-C\lambda^{-1}.
	\end{equation}

	Then we take a sequence of subsets of $\zzz$ including the boundary of $B(x_1,\lambda^{-1}m)$. 
	For each $k\ge 1$, let 
		$\vep_k=\lambda^{-\zeta_1/6}2^{-k-10},\ \eta_k=(2k)^{-1}$ and
	\[
		A_k=B(x_1,(1+\eta_k)\lambda^{-2}m)\setminus B(x_1,(1-\eta_k)\lambda^{-2}m).
	\]
	Write $k_0$ for the smallest integer satisfying $\lambda^{-2}m\vep_{k_0}<1$. 
	Note that the condition (\ref{mislarge}) guarantees that both the inner and outer boundary 
		of $B(x_1,\lambda^{-2}m)$ are contained in $A_{k_0}$. 
	Moreover, let $D_k$ be a set of lattice points in $A_k$ 
		such that $A_k\subset \bigcup_{z\in D_k} B(z,\lambda^{-2}m\vep_k)$. 
	We may suppose that $|D_k|\le C\vep_k^{-3}$. 
	Since $\lambda^{-2}m\vep_{k_0}<1$ and 
			$\partial_i B(x_1,\lambda^{-2}m)\subset A_{k_0}$, it follows that 
			$\partial_i B(x_1,\lambda^{-2}m)\subset D_{k_0}$. 

	Now we perform Wilson's algorithm to prove (\ref{smball}). 
	Let $\ust^N_0\coloneqq \xi_N\cup \gamma_{x_1}$. 
	\begin{enumerate}[(i)]
		\item Consider an independent SRW started at a point in $D_1$ and 
					run until it hits $\ust^N_0$. 
				We add its loop-erasure to $\ust^N_0$ and denote the union by $\ust^N_{1,1}$. 
				Given $\ust^N_{1,j}$, we consider an independent SRW from another point 
					in $D_1\setminus \ust^N_{1,j}$ and let $\ust^N_{1,j+1}$ 
					be the union of $\ust^N_{1,j}$ and the loop-erasure of the new SRW. 
				We continue this procedure until all points in $D_1$ are contained in the tree, 
					which we denote by $\ust^N_1$. 
		\item We repeat the above procedure for $D_2$ taking $\ust^N_1$ as a root. 
				Let $\ust^N_2$ be the output tree. 
				We continue inductively to construct $\ust^N_3,\ust^N_4,\cdots \ust^N_{k_0}$. 
		\item Once we obtain $\ust^N_{k_0}$, we perform Wilson's algorithm for all points in $B(x_1,\lambda^{-2}m)$. 
		\item We repeat the same procedure as (i), (ii) and (iii) 
					for all $x_2,x_3,\cdots x_N$. 
		\item Finally, we perform Wilson's algorithm for all points in $\bigcup_{j=0}^N \Bx{j}$ to obtain $\ust^N$. 
	\end{enumerate}
	By construction, it is clear that $\ust_k^N\subset \ust_{k+1}^N$, and 
		also $\partial_i B(x_1,\lambda^{-2}r)\subset \ust_{k_0}^N$. 

	By the definition of $G(\zeta_1)$, we have that 
	\begin{equation}\label{probH1}
		\prob(\gamma(y,\ust^N_0)\not\subset B(x_1,\lambda^{-3/2}m)\mid F\cap G(\zeta_1))\le \lambda^{-\zeta_1}.
	\end{equation}
	On the other hand, by stopping conditioning $\gamma_{x_1}$ on $F\cap G(\zeta_1)$, 
		it follows from \cite{S18}*{Theorem 1.4} and \cite{LS19}*{Corollary 1.3} that 
		there exist some universal constant $C,\ c,\ c'>0$ such that 
	\begin{align}
		\prob&\left(\gamma_{\ust^N}(y,\ust^N_0)\subset B(x_1,\lambda^{-1}m),\ d_{\ust^N}(y,\ust^N_0)\ge \frac{1}{2}\lambda^{-1} m^\beta\right)	\notag	\\
		&\le\frac{\prob(\gamma_{\ust^N}(y,\ust^N_0)\subset B(x_1,\lambda^{-1}m),\ d_{\ust^N}(y,\ust^N_0)\ge \frac{1}{2}\lambda^{-1}m^\beta)}{\prob(F\cap G(\zeta_1))}	\notag	\\
		&\le C\exp\{-c\lambda^{c'}\}.	\label{probH2}
	\end{align}
	Combining (\ref{probH1}) and (\ref{probH2}), we have that 
	\[
		\prob(\gamma(y,\ust^N_0)\subset B(x_1,\lambda^{-1}m),\ d_{\ust}(y,\ust^N_0)\le \frac{1}{2}\lambda^{-1}m^\beta)\ge 1-C\lambda^{-\zeta_1}.
	\]
	Let $H$ be the event defined by 
	\[
		H=\left\{\gamma(y,\ust^N_0)\subset B(x_1,\lambda^{-1}m),\ d_{\ust}(y,\ust^N_0)\le \frac{1}{2}\lambda^{-1}m^\beta\mbox{ for all }y\in D_1 \right\}.
	\]
	Then we have 
	\[
		\prob(H)\ge 1-C\lambda^{-\zeta_1/2},
	\]
	since $|D_1|\le C\lambda^{\zeta_1/2}$. 

	Next, we will consider several events that ensure hittability of branches 
		in the subtree. 
	For $k\ge 1$ and $\zeta>0$, we define the event $I(k,x,\zeta)$ by 
	\begin{align}\label{defhit}
		I&(k,x,\zeta)	\notag	\\
		&=\left\{P_R^y\left(R\left[0,T_{R^y}(y,\lambda^{-2}m\vep_k^{1/2})\right]\cap (\ust^N_0\cup \gamma(x,\ust^N_0))\right)\le \vep_k^\zeta \mbox{ for all }y\in B(x,\lambda^{-2}m\vep_k)\right\},
	\end{align}
	Let $I(k,\zeta)=\bigcap_{x\in D_k}I(k,x,\zeta)$. 
	Applying \cite{SaSh18}*{Lemma 3.2}, it follows that there exist universal constants 
		$\zeta_2>0$ and $C<\infty$ such that for all $k\ge 1, m\ge 1, \lambda\ge 2$ and 
		$x\in D_k$,
	\[
		\prob_{\ust^N}(I(k,x,\zeta_2)^c)\le C\vep_k^5.
	\]
	Combining this with $|D_k|\le C\vep_k^{-3}$ yields that 
	\[
		\prob(I(k,\zeta_2)^c)\le C\vep_k^2\le C\lambda^{-\zeta_1/3}.
	\]
	
	We set $A'_1\coloneqq F\cap G(\zeta_1)\cap H\cap I(1,\zeta_2)$. 
	Note that $A'_1$ is measurable with respect to $\ust^N_1$, 
		the subtree obtained after the first step (i) of our Wilson's algorithm. 
	We have already seen that $\prob_{\ust^N}(A'_1)\ge 1-C\lambda^{-\zeta_1/3}$. 

	Conditioning $\ust^N_1$ on the event $A'_1$, we proceed with Wilson's algorithm 
		for the points in $D_2$. 
	We take $y\in D_2$ and consider the SRW $R^y$ started at $y$ until it hits $\ust^N_1$.  
	By the definition of $D_1$, there exists $x'\in D_1$
		 such that $d_E(x,y)\le \lambda^{-2}m\vep_1$. 
	Suppose that $R^y$ exits $B(y,\lambda^{-2}m\vep_1^{1/3})$ before it hits $\ust^N_1$. 
	Then the event that $R^y$ exits $B(x,\lambda^{-2}m\vep_1)$ before it hits 
		$\ust^N_1$ occurs. 
	However by (\ref{defhit}) and the definition of $\zeta_2$, 
		the probability that the event occurs conditioned on $A'_1$ is lower than 
		$\vep_1^{\zeta_2}$. 
	By iteration, the number of balls of radius $\lambda^{-2}m\vep_1^{1/2}$ that $R^y$ 
		exits before hitting $\ust^N_1$ is larger than $\vep_1^{-1/6}$. 
	Hence, we have that 
	\[
		P^y(R^y\mbox{ exits }B(y,\lambda^{-2}m\vep_1^{1/3})\mbox{ before it hits }\ust^N_1)\le \vep_1^{c\zeta_2\vep_1^{-1/6}},
	\] 
	for some universal constant $c>0$. 
	Moreover, following the same argument as (\ref{probH2}), we have that 
	\[
		P^y\left(\gamma(y,\ust^N_1)\not\subset B(y,\lambda^{-2}m\vep_1^{1/3})\mbox{ and }d_{\ust}(y,\ust^N_1)\ge (\lambda^{-2}m)^\beta\vep_1^{1/4}\right) \le C\exp\left\{-c\vep_1^{-1/12}\right\}.
	\]
	With this in mind, we define the event $B_2$ by 
	\[
		B_2=\left\{\gamma(y,\ust^N_1)\subset B(y,\lambda^{-2}m\vep_1^{1/3})\mbox{ and }d_{\ust}(y,\ust^N_1)\le \lambda^{-1}m^\beta\vep_1^{1/4},\mbox{ for all }y\in D_2\right\}.
	\]
	Since $|D_2|\le C\vep_2^{-3}$, we have that 
	\[
		\prob_{\ust^N}(B_2\mid A'_1)\ge 1-C\vep_1^{-3}\exp\left\{-c\vep_1^{-1/12}\right\}.
	\]
	Hence, letting $A'_2\coloneqq A_1\cap B_2\cap I(2,\zeta_2)$, it follows that 
	\[
		\prob_{\ust^N}(A'_2\mid A'_1) \ge 1-C\vep_2^2.
	\]
	Following the above argument, we define the sequences of events $\{A'_k\},\{B_k\}(k=2,3,\cdots,k_0)$ by 
	\begin{align*}
		B_k&=\left\{\gamma(y,\ust^N_{k-1})\subset B(y,\lambda^{-2}m\vep_{k-1}^{1/3})\mbox{ and }d_{\ust}(y,\ust^N_{k-1})\le \lambda^{-1}m^\beta\vep_{k-1}^{1/4},\mbox{ for all }y\in D_2\right\},	\\
		A'_k&=A'_{k-1}\cap B_k\cap I(k,\zeta_2).
	\end{align*}
	Then we can conclude that 
	\begin{equation}\label{Ak0bound}
		\prob_{\ust^N}(A'_{k_0})=\prob_{\ust^N}(A'_1)\prod_{k=2}^{k_0}\prob_{\ust^N}(A'_k\mid A'_{k-1})\ge (1-C\lambda^{-\zeta_1/3})\prod_{k=1}^\infty(1-C\vep_k^2)\ge 1-C\lambda^{-\zeta_1/3}.
	\end{equation}
	On the other hand, on the event $A'_{k_0}$, there exists some universal constant 
		$C>0$ such that
	\begin{enumerate}[(1)]
	\setlength{\leftskip}{2em}
		\item $d_{\ust}(x_1,y)\le \lambda^{-1}m^\beta$ for all $y\in\left(\ust^N_0\cap B(x_1,\lambda^{-2}m)\right)\cup\left(\bigcup_{y\in D_1}\gamma(y,\ust^N_0)\right)$,
		\item $d_{\ust}(x_1,y) \le C\lambda^{-1}m^\beta$ for all $y\in (\ust^N_0\cap B(x_1,\lambda^{-2}m))\cup\ust^N_{k_0}$,
	\end{enumerate}
	It immediately follows that (1) holds from the definition of $F$ and $H$. 
	For $y\in\ust_{k_0}$, let $y_k~(k=1,2,\cdots k_0-1)$ be the first point in $\ust^N_{k}$ 
		that appears on $\gamma_{\ust^N}(y,\ust^N_0)$ (we set $y_k=y$ if $y\in\ust^N_k$). 
	On the event $A'_{k_0}$, we have that 
	\begin{align*}
		d_\ust(x_1,y)&\le d_\ust(x_1,y_1)+\sum_{k=1}^{k_0-1}d_\ust(y_k,y_{k+1})	\\
				&\le \lambda^{-1}m^\beta+\sum_{k=1}^\infty \lambda^{-1}m^\beta\vep_{k-1}^{1/4} \le C\lambda^{-1}m^\beta,
	\end{align*}
	which implies (2). 

	Once we see that (2) holds on the event $A'_{k_0}$, 
		we need to estimate the $d_\ust$ distance between
		an arbitrary point in $B(x_1,\lambda^{-2}m)$ and $\partial_i B(x_1,\lambda^{-2}m)$. 
	In order to do so, we take another ``net'':
		 we let $\vep'_k=\lambda^{-\zeta_1/4}2^{-k-10}$ and 
		$D'_k$ be a set of lattice points in $B(x_1,\lambda^{-2}m)$ such that 
			$B(x_1,\lambda^{-2}m)\subset \bigcup_{z\in D'_k}B(z,\lambda^{-2}m\vep'_k)$. 
	We may suppose that $|D'_k|\le C(\vep'_k)^{-3}$. 
	By the similar argument to the estimate of $\prob_{\ust^N}(A'_{k_0})$, we obtain 
	\begin{equation}\label{filll}
		\prob_{\ust^N}\left(d_\ust(y,\partial_iB(x_1,\lambda^{-2}m))\ge C\lambda^{-1}m^\beta\mbox{ for some }y\in B(x_1,\lambda^{-2}m)\right)\le C\lambda^{-\zeta_1/2}.
	\end{equation}

	For the lower bound of volume (\ref{smball}), we now estimate the distance between $x_1$
		and all points in $B(x_1,\lambda^{-2}m)$. 
	Since $\ust^N_{k_0}$ contains $\partial_i B(x_1,\lambda^{-2}m^\beta)$, 
		it follows from the same argument as (\ref{probH2}) again that 
		for any $y\in B(x,\lambda^{-2}m)$,  
	\begin{align*}
		\prob_{\ust^N}&\left(d_{\ust}(x_1,y) \le C\lambda^{-1}m^\beta \mbox{ for all }y\in B(x_1,\lambda^{-2}m)\right)	\\
		&\ge \prob_{\ust^N}(A'_{k_0})-\prob_{\ust^N}\left(A'_{k_0}\cap\left\{d_{\ust}(x_1,y)>C\lambda^{-1}m^\beta\mbox{ for some }y\in B(x_1,\lambda^{-2}m)\right\}\right)	\\
		&\ge \prob_{\ust^N}(A'_{k_0})-\prob_{\ust^N}\left(d_{\ust}(y,\partial_iB(x_1,\lambda^{-2}m))>C\lambda^{-1}m^\beta\mbox{ for some }y\in B(x_1,\lambda^{-2}m)\right)	\\
		&\ge 1-C\lambda^{-c},
	\end{align*}
	for some universal constant $c>0$, which completes the proof of (\ref{smball}). 
	\end{proof}
	
	It follows from (\ref{onebox}) and (\ref{smball}) that there exists some 
	universal constant $\lambda_*\ge 1$ such that 
	\begin{equation}\label{abefm_j}
		P(A_j\cap B_j\cap E_j(C_*)\cap F_j(\eta_*)\cap M_j(\lambda_*))=c_*^2N^{-2},
	\end{equation}
	for $j=1,2,\cdots N$. 
	We have obtained a lower bound of the probability that the volume of the random tree 
		constructed by Wilson's algorithm in each $\Bx{j}$ is of the order of $m^\beta$. 
	
	Recall that the events $G_j$ and $H_j$ are defined by (\ref{G_j}) and (\ref{H_j}) 
		respectively. 
	We define an event $I_j$ by $I_j=\bigcap_{k=1}^j M_k(\lambda_*)$.  
	Let $J_j=\bigcap_{k=1}^j E_k(C_*)$ and $K_j=\bigcap_{k=1}^j F_k(\eta_*)$ 
			for $C_*$ and $\eta_*$ defined in Lemma \ref{c-eta}, where $E_k(C)$ and $F_k(\eta)$ are as defined in 
			(\ref{defE}) and (\ref{defF}), respectively. 
	Recall that $T_S(r)=\inf\{k\ge 0\mathrel{:}|S(k)|\ge r\}$. 
	
	By (\ref{2len}), in order to estimate $\mathrm{len}(\xi_j)$ 
		on the event $G_j\cap H_j\cap I_j\cap J_j\cap K_j$, 
	we need to give an upper bound on $|\xi_j\cap Q[a_k-q,a_k+q]|$ for $k=1,\cdots,j$ and 
		for $q$ defined in (\ref{defofq}). 
	Take 
	\begin{equation}\label{defR(N)}
		w=(w^1,w^2,w^3),\ \ \ R(N)=\exp\{2e^{RN^2}+1\},
	\end{equation}
	and define 
	\begin{align*}
		Q_w&=\{y=(y^1,y^2,y^3)\mathrel{:}|y^1-w^1|\le q,\ |y^i-w^i|\le m/2\ \mbox{for }i=2,3\},	\\
		N_w&=\left|Q_w\cap \mathrm{LE}(S[0,T_S(R(N)m)])\right|.
	\end{align*}	
	
	\begin{figure}[htb]
		\centering
		\includegraphics[width=0.56\linewidth]{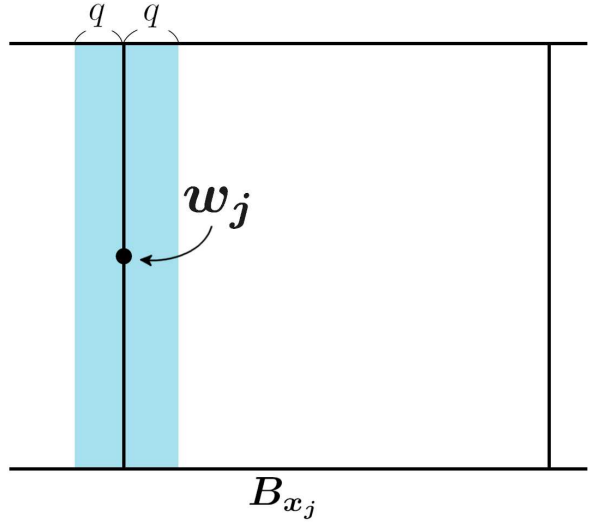}\label{figureQw}
		\caption{$Q_w$ for $w=w_j=((j-\frac{1}{2})m,0,0)$}
	\end{figure}
	
	Then, it follows from \cite{LS21}*{Lemma 4.5} that there exist universal constants 
		$0<c,C<\infty$ such that 
	\begin{equation}\label{N_w}
		P(N_w\ge m^\beta)\le C\exp\{-cN^2\}\ \ \mbox{uniformly in }w\in B(0,R(N)m).
	\end{equation}
	
	We define
	\begin{align*}
		w_j=&(a_j,0,0),\ \ \ L_j=\left\{|Q_{w_k}\cap \xi_j|\le m^\beta\ \mbox{for all }1\le k\le j/2\right\}\ \ \mbox{for }j\ge 1,	\\
		U_{2N}=&\left\{S(T_S(R(N)))\in\{(y^1,y^2,y^3)\in\rrr\mathrel{:} y^1\ge \frac{4}{5}R(N)m\},\right. 	\\
		&\ \ \left.S[\tS{2N+1},T_S(R(N))]\cap B(0,a_{\frac{7}{4}N})=\emptyset\right\}.
	\end{align*}
	and set 
	\begin{equation}\label{defofAN}
		A^N=G_{2N}\cap H_{2N}\cap I_{2N}\cap J_{2N}\cap K_{2N}\cap L_{2N}\cap U_{2N}.
	\end{equation}
	
	\begin{figure}[htb]
		\centering
		\includegraphics[width=0.8\linewidth]{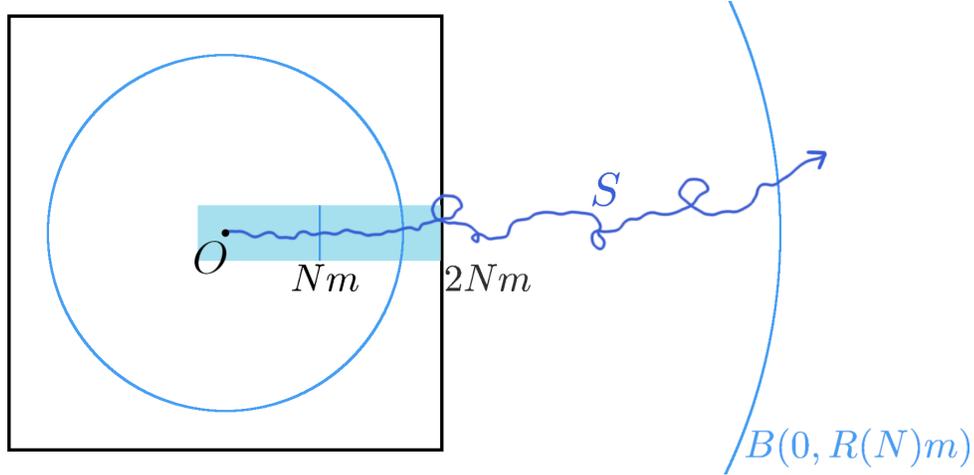}\label{eventU2N}
		\caption{Definition of the event $U_{2N}$}
	\end{figure}
	
	We estimate the lower bound of the probability of the event $A^N$. 
	
	\begin{lem}\label{x1axis}
	There exists a universal constant $c_3>0$ such that 
	\begin{equation}
			P(A^N)\ge c_3^{-1}\exp \{-c_3N(\log N)\}
	\end{equation}
	\end{lem}
	
	\begin{proof}	
		To prove this, we will make use of the strong Markov property of $S$ as follows. 
		First, by the strong Markov property
		\begin{align*}
			P(G_{2N}&\cap H_{2N}\cap I_{2N}\cap J_{2N}\cap K_{2N})	\\	
				&=P((A_{2N}\cap B_{2N}\cap E_{2N}\cap F_{2N}\cap M_{2N})\cap (G_{2N-1}\cap H_{2N-1}\cap I_{2N-1}\cap J_{2N-1}\cap K_{2N-1}))	\\
				&=P^{S(t_S(a_{2N}))}(A_{2N}\cap B_{2N}\cap E_{2N}\cap F_{2N}\cap M_{2N})P(G_{2N-1}\cap H_{2N-1}\cap J_{2N-1}\cap K_{2N-1}).
		\end{align*}
		Then by (\ref{abefm_j}), we have that 
		\[
			P(G_{2N}\cap H_{2N}\cap I_{2N}\cap J_{2N}\cap K_{2N})\ge c_*N^{-2}P(G_{2N-1}\cap H_{2N-1}\cap I_{2N-1}\cap J_{2N-1}\cap K_{2N-1}),
		\]
		and by iteration, it follows that there exists some universal constant $c>0$ such that 
		\begin{equation}\label{ghjk2N}
			P(G_{2N}\cap H_{2N}\cap I_{2N}\cap J_{2N}\cap K_{2N})\ge (cN^{-2})^{2N}.
		\end{equation}
		Secondly, again by the strong Markov property 
		\begin{align*}
			P(G_{2N}&\cap H_{2N}\cap I_{2N}\cap J_{2N}\cap K_{2N}\cap U_{2N})	\\		
				=&P(U_{2N}\mid G_{2N}\cap H_{2N}\cap I_{2N}\cap J_{2N}\cap K_{2N})P(G_{2N}\cap H_{2N}\cap J_{2N}\cap K_{2N})	\\
				=&P^{S(t_S(a^{2N+1}))}\left(U_{2N}\right)P(G_{2N}\cap H_{2N}\cap I_{2N}\cap J_{2N}\cap K_{2N}).
		\end{align*}
		Then by \cite{L96}*{Proposition 1.5.10}, $P^{S(t_S(a^{2N+1}))}\left(U_{2N}\right)$ is 
			bounded below by some universal constant $c>0$. 
		Combining this with (\ref{ghjk2N}), we obtain   
		\begin{equation}\label{G-U}
			P(G_{2N}\cap H_{2N}\cap I_{2N}\cap J_{2N}\cap K_{2N}\cap U_{2N}) \ge c\exp\{-cN(\log N)\}.
		\end{equation}
		Furthermore, following the proof of \cite{LS21}*{Proposition 4.6}, 
			we obtain that 
		\[
			P(G_{2N}\cap H_{2N}\cap I_{2N}\cap J_{2N}\cap K_{2N}\cap U_{2N}\cap (L_{2N})^c)\le CN\exp\{-cN^2\},
		\]
		where we use (\ref{N_w}) instead. 
		Combining this with (\ref{G-U}), we obtain 
		\[
			P(G_{2N}\cap H_{2N}\cap I_{2N}\cap J_{2N}\cap K_{2N}\cap L_{2N}\cap U_{2N})\ge c\exp\{-cN(\log N)\},
		\]
		which completes the proof. 	
	\end{proof}



\subsection{Proof of Proposition \ref{VLower}}
	In the previous subsection, we observed the behavior of $\ust$ along the LERW started at 
		the origin, \ie the first step of Wilson's algorithm. 
	Now we consider several events to complete a lower bound estimate of the upper tail of the 
		volume. 
	
	We take $y_j=(jm,Nm,0)\in\zzz,\ j=1,\cdots,N,$ and run a simple random walk $R^{y_j}$ 
		started at $y_j$ until it hits $\ust^N$. 
	Let $B_{y}(k)$ be the cube of side-length $m$ centered at $y_{j,k}=(jm,km,0)\in\zzz$. 
	We define $\tau_{y_j}$ (resp. $\sigma_{y_j}$) to be the first time that $R^{y_j}$ hits 
		$B_{y}(0)=\Bx{j}$ (resp. $\ust^N$).
 
	\begin{defi}\label{parallel}
		Define $V_{y_j}$ to be the intersection of the following events of $R^{y_j}$: 
		\begin{itemize}
			\item $\{\mathrm{LE}(R^{y_j}[0,\tau_{y_j}])\subset \bigcup_{k=1}^{N-1} B_{y}(k)\}$, 
			\item $\{\mathrm{len}(\mathrm{LE}(R^{y_j}[0,\tau_{y_j}]))\le (N-1)m^\beta\}$,
			\item $\displaystyle \bigcap_{k=1}^{N-1} \left\{\forall z\in B(y_{j,k},2\lambda^{-2}m),\ P_{\widetilde{R}}^z([0,T_{\widetilde{R}^z}(y_{j,k},\lambda^{-3/2}m)]\cap \mathrm{LE}(R^{y_j}[0,\tau_{y_j}] =\emptyset)\le \lambda^{-\zeta_1}\right\}$, where $\widetilde{R}$ is a simple random walk independent of $R^{y_j}$ and $\zeta_1$ is as defined in (\ref{zeta_1}),
			\item $\bigcap_{k=1}^{N-1} \{B(y_{j,k},\lambda^{-2}m)\subset B_{\ust}(y_{j,k},\lambda^{-1}m^\beta)\}$,
		\end{itemize}
		and define $W_{y_j}$ by 
		\begin{equation*}
			W_{y_j}\coloneqq \left\{\mathrm{len(LE}(R^{y_j}[\tau_{y_j},\sigma_{y_j}]))\ge cN(\log N)^{100}m^\beta\right\}
		\end{equation*}
	\end{defi}
	\begin{figure}[htb]
		\centering
		\includegraphics[width=0.8\linewidth]{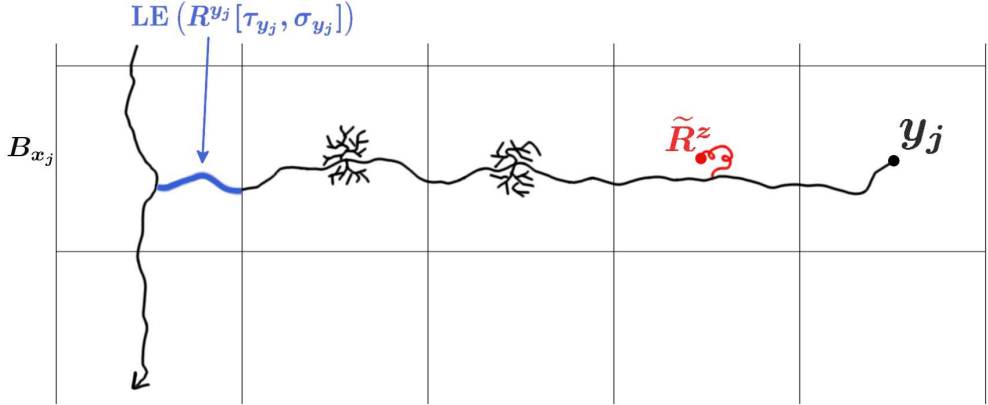}\label{VWy_j}
		\caption{The sets we consider in the event $V_{y_j}\cap W_{y_j}$}
	\end{figure}

	\begin{lem}\label{yz}
		For each $j=1,2,\cdots N$,
		\begin{equation}\label{probvwy}
			P^{y_j}(V_{y_j}\cap W_{y_j}\mid A^N)\ge c\exp\{-CN(\log N)\}.
		\end{equation}
	\end{lem}
	
	\begin{remark}
		Since tail bounds for the length of three-dimensional LERW in a general set have not been 
			obtained, we will apply the tail bound for the length of infinite LERW in a Euclidean ball 
			given in \cite{S18}*{Theorem 1.4} and \cite{LS19}*{Corollary 1.3}, 
			instead of regarding $\gamfty$ as a deterministic set. 
		Thus, in order to estimate the conditional probability of $W_{y_j}$ on the event $A^N$ 
			from below, we consider the length $cN(\log N)^{100}m^\beta$ 
			in the right-hand side of the definition of $W_{y_j}$, so that $P^{y_j}(W_{y_j}^c)$ 
			becomes enough compared to $\prob(A^N)$. 
	\end{remark}
	
	\begin{proof}
		First, applying the same argument as Lemma \ref{balls} and Lemma \ref{x1axis}, 
			there exists a universal constant $c>0$ such that 
			$P^{y_j}(V_{y_j}\mid A^N)\ge c\exp\{-cN\log N\}$. 
		Next we will estimate the upper bound of $P^{y_j}(A^N\cap V_{y_j}\cap W_{y_j}^c)$. 
		In order to do so, we stop conditioning on $A^N$ and consider 
			$\mathrm{LE}(R^{y_j}[\tau_{y_j},\sigma_{y_j}])$ as a part of infinite LERW. 
		By \cite{S18}*{Theorem 1.4} and \cite{LS19}*{Corollary 1.3}, we have that 
		\begin{align*}
			P^{y_j}(A^N\cap V_{y_j}\cap W_{y_j}^c)&\le P\left(\mathrm{len(LE}(R^{y_j}[\tau_{y_j},\sigma_{y_j}]))\ge cN(\log N)^{100}m^\beta\right)	\\
				&\le \exp\{-CN(\log N)^{100}\},
		\end{align*}
		from which it follows that 
		\[
			P^{y_j}(A^N\cap V_{y_j}\cap W_{y_j})=P^{y_j}(A^N\cap V_{y_j})-P^{y_j}(A^N\cap V_{y_j}\cap W_{y_j}^c)\ge \frac{1}{2}P^{y_j}(A^N\cap V_{y_j}),
		\]
		since $P(F_0)\ge c\exp\{-CN\log N\}$ by Lemma \ref{x1axis} and Lemma \ref{balls}. 
		Thus, we have
		\begin{align*}
			P^{y_j}(V_{y_j}\cap W_{y_j}\mid A^N)&\ge \frac{1}{2}P^{y_j}(V_{y_j}\mid A^N)	\\
				&\ge c\exp\{-cN(\log N)\},
		\end{align*}
		which completes the proof.
	\end{proof}
	Finally we take $z_{j,k}=(jm,km,Nm)\in\zzz,\ j,k=1,\cdots,N,$ and run a simple random walk 
		$R^{z_{j,k}}$ started at $z_{j,k}$ until it hits $\ust^N$. 
	Let $B_{z}(l)$ be the cube of side-length $m$ centered at $z_{j,k,l}=(jm,km,lm)\in\zzz$. 
	We define $\tau_{z_{j,k}}$ (resp. $\sigma_{z_{j,k}}$) to be the first time that 
		$R^{z_{j,k}}$ hits $B_{z}(0)=B_{y}(k)$ (resp. already constructed subtree of $\ust$). 
	Let $V_{z_{j,k}}$ (resp. $W_{z_{j,k}}$) be an event of $R^{z_{j,k}}$ which corresponds 
		to $V_{y_j}$ (resp. $W_{y_j}$) with $x_2$ axis replaced by $x_3$ axis 
		(see Definition \ref{parallel} for the definition of $V_{y_j}$ and $W_{y_j}$). 
	By applying the same argument as Lemma \ref{yz}, we have that 
	\begin{equation}\label{probvwzjk}
		P^{z_{j,k}}(V_{z_{j,k}}\cap W_{z_{j,k}}\mid A^N\cap V_{y_j}\cap W_{y_j})\ge c\exp\{-CN(\log N)\}.
	\end{equation}
	Note that $V_{z_{j,k}}\cap W_{z_{j,k}}$ is independent of 
		$V_{y_{j'}}\cap W_{y_{j'}}$ if $j\neq j'$.

	\begin{cor}\label{corvvv}
	There exist universal constants $c,c',C,C'$ such that for all $m\ge 1$ and $N\ge 1$, 
	\begin{equation}\label{vvv}
		\prob\left(|B_\ust(0,C'Nm^\beta)|\ge c'(Nm)^3\right) \ge c\exp\{-CN^3(\log N)\}.
	\end{equation}
	\end{cor}
	
	\begin{proof}
	Recall that $x_j=(jm,0,0)$. 
	On the event $A^N$, we have that for all $y\in B(x_j,\lambda_*^{-2}m)$, 
	\begin{align*}
		d_\ust(0,y)&\le d_\ust(0,x_j)+d_\ust(x_j,y)	\\
				&\le CNm^\beta+\lambda_*^{-1}m^\beta	\\
				&\le C'Nm^\beta
	\end{align*}
	Since each step of Wilson's algorithm is mutually independent, applying the result of 
		Lemma \ref{yz} to the ``tubes'' parallel to $x_2$ axis, we obtain 
	\begin{align*}
		\prob&\left(\bigcap_{k=0}^N \bigcap_{j=1}^N \left\{d_\ust(0,y)\le CNm^\beta\mbox{ for all }y\in B\left((jm,km,0),\lambda_*^{-2}m\right)\right\}\right)	\\
			&\ge \prob(A^N)\prod_{j=1}^{N}P^{y_j}(V_{y_j}\cap W_{y_j}\mid A^N)	\\
			&\ge c\exp\{-CN^2(\log N)\}.
	\end{align*}
	Next we consider the ``tubes'' parallel to $x_3$ and we have 
	\begin{align*}
		\prob&\left(\bigcap_{l=0}^N\bigcap_{k=0}^N \bigcap_{j=1}^N \left\{d_\ust(0,y)\le CNm^\beta\mbox{ for all }y\in B\left((jm,km,lm),\lambda_*^{-2}m\right)\right\}\right)	\\
			&\ge c\exp\{-CN^2(\log N)\}\prod_{k=1}^N \prod_{j=1}^N P^{z_{j,k}}\left(V_{z_{j,k}}\cap W_{z_{j,k}}\mid A^N\cap V_{y_j}\cap W_{y_j}\right)	\\
			&\ge c\exp\{-CN^3(\log N)\},
	\end{align*}
	where we applied (\ref{probvwzjk}) in the last inequality. 
	Finally, comparing the left-hand side of the above inequality and (\ref{vvv}), we obtain (\ref{vvv}). 
	\end{proof}

	\begin{proof}[Proof of Propositon \ref{VLower}]
		Setting $r=C'Nm^\beta$ and $\lambda=c'N^{3(\beta-1)/\beta}/C'^{3/\beta}$ 
			in (\ref{vvv}) yields the result at (\ref{VLower}).
	\end{proof}
	
	\begin{theo}\label{thmvup}
		$\prob$-a.s.,
		\begin{equation}\label{vlowfluc3}
			\limsup_{r\to\infty}(\log\log{r})^{-(\beta-1)/\beta}r^{-3/\beta}|B_\ust(0,r)|=\infty.
		\end{equation}
	\end{theo}
	
	\begin{proof}
	First we define a sequence of scales. Fix $\vep>0$ and let 
	\[
		D_i=e^{i^2},\ m_i=D_i/\vep(\log i)^{1/3}.
	\]	
	We now run Wilson's algorithm. 
	Let $\gamfty$ be the infinite LERW started at the origin and 
		$(S^z)_{z\in\zzz}$ be the family of independent SRW which is also independent of $\gamfty$. 
	First, at stage $i~(i\ge 1)$, we use all the vertices in $B_\infty(0,D_i)$ 
		which have not already been contained and write $\ust_i$ for the tree obtained. 
	
	By \cite{ACHS20}*{Proposition 4.1}, there exists $M>0$ such that the event 
		\begin{equation}\label{inc1}
			B_\infty(0,D_i)\subset B_\ust(0,i^M D_i^\beta)\subset B_\infty(0,i^{2M}D_i)
		\end{equation}
	occurs with probabilty greater than $1-ci^{-2}$. 
	Hence, if we run Wilson's algorithm for the vertices contained in $B_\infty(0,D_i)$ taking $\gamfty$ as the root, 
		then the probability that $\ust_i$ leaving $B_\infty(0,i^{2M})$ is less than $c\lambda^{-2}$. 
	By applying Borel-Cantelli lemma, we obtain that
	\begin{equation}\label{inc2}
		\ust_i\subset B_\infty(0,i^{2M}D_i)\subset B_\infty(0,m_{i+1}/2)
	\end{equation}
	for large $i$, almost-surely. 
	Moreover, from (\ref{inc1}), we may also assume that
	\begin{equation}\label{inc3}
		d_\ust(0,z)\le i^MD_i^\beta\le m_{i+1}^\beta\ \ \ \mbox{for all}\ \ \ z\in\ust_i
	\end{equation}
	almost-surely. 

	Define the event $F(i)$ to be the event that both (\ref{inc2}) 
		and (\ref{inc3}) hold. 
	Let $\mathcal{F}_i$ be the $\sigma$-field generated by the followings: 
	\vspace{-0.5\baselineskip}
	\begin{itemize}
	\setlength{\parskip}{0cm}
 	\setlength{\itemsep}{0cm}
		\item $\gamfty[0,\tau_{B_\infty(0,i^{2M}D_i})]$,  
		\item All simple random walks added to $\ust_{i-1}$ at stage $i$, 
	\end{itemize}
	\vspace{-0.5\baselineskip}
	where $\tau_A$ represents the first exiting time from $A$. 
	
	Now we bound the probability that the subtree $\ust_{i+1}$ obtained at stage $i+1$ 
		also satisfies the diameter estimate and the inclusion corresponding to (\ref{inc2}) 
		conditioned that $F(i)$ holds. 
	We define an event $G(i)$ by
	\begin{align*}
		G(i)&=\{|\gamfty[0,\tau_{B_\infty(0,m_{i})}]|\le m_{i}^\beta\}\cap A^{\vep(\log i)^{1/3}}\cap \left(\bigcap_{j=1}^{\vep(\log i)^{1/3}} V_{y_j}\cap W_{y_j}\right) \\
		&\cap \left(\bigcap_{k=1}^{\vep(\log i)^{1/3}}\bigcap_{j=1}^{\vep(\log i)^{1/3}} V_{z_{j,k}}\cap W_{z_{j,k}}\right),
	\end{align*}
	where replace the scales $m$ and $N$ by $m_i$ and $\vep(\log i)^{1/3}$, respectively. 
	See (\ref{defofAN}) and Definition \ref{parallel} for the definition of the events 
		$A^N$, $V_{y_j}\cap W_{y_j}$ and $V_{z_{j,k}}\cap W_{z_{j,k}}$. 

	For $A\subset\zzz$, let 
	\[
		\tau'_A=\sup\{i\mathrel{:}\gamfty(i)\in A\},
	\]
	be the last time that $\gamfty$ exits from $A$ and 
		recall that $\tau_A$ indicates the first exiting time. 
	Then, by \cite{M09}*{Proposition 4.6}, $\gamfty[0,\tau_{B_\infty(0,i^{2M}D_i})]$ and 
		$\gamfty[\tau'_{B_\infty(0,i^{4M}D_i)},\infty)$ is ``independent up to constant'', 
		\ie there exists a universal constant $C>0$ such that for any $i$ and 
		any possible paths $\eta_1,\eta_2$,
	\begin{align}\label{indp_const}
			\prob(\gamfty[0,\tau_{B_\infty(0,i^{2M}D_i)}]&=\eta_1,\gamfty[\tau'_{B_\infty(0,i^{4M}D_i)},\infty)=\eta_2)	\notag		\\
			&\ge C\prob(\gamfty[0,\tau_{B_\infty(0,i^{2M}D_i)}]=\eta_1)
				\prob(\gamfty[\tau'_{B_\infty(0,i^{4M}D_i)},\infty)=\eta_2).
	\end{align}
	Let $\widehat{\gamfty}=\gamfty[\tau_{B_\infty(0,i^{2M}D_i)},\tau_{B_\infty(0,i^{4M}D_i)}]$. Then, 
	\begin{align}
		\prob&(G(i+1)\mid F_i)	\notag	\\
		&\ge \prob\left( A^{\vep(\log (i+1))^{1/3}}\cap \left(\bigcap_{j=1}^{\vep(\log (i+1))^{1/3}} V_{y_j}\cap W_{y_j}\right)\cap \left(\bigcap_{k=1}^{\vep(\log (i+1))^{1/3}}\bigcap_{j=1}^{\vep(\log (i+1))^{1/3}} V_{z_{j,k}}\right)\right)	\notag	\\
		&-\prob(|\widehat{\gamfty}|\ge m_{i+1}^\beta \mid F_i).	\label{hikizan}
	\end{align}
	By Corollary \ref{corvvv}, we have that the first term of (\ref{hikizan}) is bounded below 
		by $Ci^{-c\vep^3}$. 

	For the second term of (\ref{hikizan}), we first consider the diameter of $\widehat{\gamfty}$.
	 
	Let $\theta_1=\{\theta_1(0),\cdots,\theta_1(k)\}$ be a path which satisfies 
	\[
		\theta_1(0)\in \partial_i B_\infty(0,i^{6M}D_i),\ \ \theta_1(k)\in \partial_i B_\infty(0,i^{4M}D_i),\ \ \theta_1(1),\cdots, \theta_1(k-1)\in B_\infty(0,i^{4M}D_i)^c,
	\]
	and $\theta_2=\{\theta_2(0),\cdots,\theta_2(l)\}$ be a path in 
		$B_\infty(0,i^{2M}D_i)$ which satisfies $\theta_2(0)=0$ and 
		$\theta_2(l)\in \partial_i B_\infty(0,i^{2M}d_i)$. 
	Let $X$ be a random walk on $\zzz$ started at $z\in \partial_i B(0,i^{6M}D_i)$ and 
		conditioned not to hit $\theta_2$. 
	We define $\sigma$ to be the first hitting time of $B_\infty(0,i^{4M}D_i)$. 
	Then by calculation of conditional probability, we have that 
	\begin{align}
		P^z(X[0,\sigma]=\theta_1)&=\frac{P^z(S[0,\sigma]=\theta_1,S[\sigma,\infty)\cap \theta_2=\emptyset)}{P^z(S[0,\infty)\cap \theta_2=\emptyset)}	\notag	\\
			&=P^z(S[0,\sigma]=\theta_1)\frac{P^{\theta_1(k)}(R[0,\infty)\cap\theta_2=\emptyset)}{P^z(S[0,\infty)=\emptyset)},	\notag	\\
	\end{align}
	where we applied the strong Markov property for the second equality. 
	Since $\theta_2$ is included in $B_\infty(0,i^{2M}D_i)$, it follows from 
	\cite{L96}*{Proposition 1.5.10} that 
	there exists some universal constant $C>0$ such that for any $\theta_1$ and $\theta_2$, 
	\begin{equation*}
		\frac{1}{C}\le \frac{P^z(X[0,\sigma]=\theta_1)}{P^z(S[0,\sigma]=\theta_1)}\le C.
	\end{equation*}
	Thus, we have 
	\begin{align}
		P(\widehat{\gamfty}\cap B_\infty(0,i^{6M}D_i)\neq \emptyset)	
		&\le \max_{z\in \partial_iB(0,i^{6M}D_i)}P(X[0,\infty)\cap B_\infty(0,i^{4M}D_i)\neq\emptyset)	\notag	\\
		&\le C\max_{z\in \partial_iB(0,i^{6M}D_i)} P(S[0,\infty)\cap B_\infty(0,i^{4M}D_i)\neq\emptyset)	\notag.
	\end{align}
	By applying \cite{L96}*{Proposition 1.5.10} again, we obtain that 
	\begin{equation}\label{widehat1}
		P(\widehat{\gamfty}\cap B_\infty(0,i^{6M}D_i)\neq\emptyset)\le Ci^{-2M}.
	\end{equation}
	
	On the other hand, we have that 
	\begin{align}
		P&(\widehat{\gamfty}\cap B_\infty(0,i^{6M}D_i)=\emptyset, |\widehat{\gamfty}|\ge m_{i+1}^\beta\mid F_i)	\notag	\\
		&\le \frac{P(\widehat{\gamfty}\cap B_\infty(0,i^{6M}D_i)=\emptyset, |\widehat{\gamfty}|\ge m_{i+1}^\beta)}{P(|\gamfty[0,\tau_{B_\infty(0,i^{2M}D_i)}]|\le m_i^\beta)}	\notag
	\end{align}
	By applying Lemma \ref{x1axis}, the denominator is bounded below by 
	\begin{equation*}
		P(|\gamfty[0,\tau_{B_\infty(0,i^{2M}D_i)}]|\le m_i^\beta)\ge P(A^{\vep(\log i)^{1/3}})\ge i^{-c\vep^3}.
	\end{equation*}
	For the numerator, now we stop conditioning on $\ust_i$ and 
	consider $\widehat{\gamfty}$ as a subset of the infinite LERW started at the origin. 
	Thus, again by \cite{S18}*{Theorem 1.4} and \cite{LS19}*{Corollary 1.3}, 
	we have that 
	\begin{align}
		P(\widehat{\gamfty}\cap B_\infty(0,i^{6M}D_i)=\emptyset, |\widehat{\gamfty}|\ge m_{i+1}^\beta)&\le P(|\gamfty[0,\tau_{B_\infty(0,i^{6M}D_i)}]|\ge m_{i+1}^\beta)	\notag	\\
			&\le Ce^{-2i}.	\notag
	\end{align}
	It follows from the above inequalities that 
	\begin{equation}\label{widehat2}
		P(\widehat{\gamfty}\cap B_\infty(0,i^{6M}D_i)=\emptyset, |\widehat{\gamfty}|\ge m_{i+1}^\beta\mid F_i)\le Ci^{c\vep^3}e^{-2i}.
	\end{equation}
	Substituting (\ref{widehat1}) and (\ref{widehat2}) into (\ref{hikizan}) yields 
	\begin{equation}
		P(G(i+1)\mid F_i)\ge Ci^{-c\vep^3},
	\end{equation}
	from which it follows that 
	\begin{align*}
		\prob(G(i+1)\mid\mathcal{F}_i)&\ge \prob(G(i+1)\mid\mathcal{F}_i)\mathbf{1}_{F(i)}	\\
		&\ge \exp\{-c(D_{i+1}/m_{i+1})^3 \log (D_{i+1}/m_{i+1})\}-Ci^{2M}	\\
		&\ge i^{-c\vep^3}
	\end{align*}
	for large $i$. Since $G(i)$ is $\mathcal{F}_i$-measureable, it follows 
		from the conditional Borel-Cantelli lemma that $G(i)$ occurs infinetely often, almost-surely. 
	Note that on $G(i)$ we have that
	\[
		|B_\ust(0,C'D_i m_i^{\beta-1})|\ge c'D_i^3.
	\] 
	Finally, the reparameterization $r_i=C'D_im_i^{\beta-1}$ yields the result. 
	\end{proof}
	

\section{Lower volume fluctuations and resistance estimate}



	In this section, we prove bounds for the volume and effective resistance which are key ingredient 
		of the proof of upper fluctuations for the heat kernel. 
	
\subsection{Lower volume fluctuations}

	In this subsection, we prove (\ref{main-2-1}), lower volume fluctuations of log-logarithmic 
		magnitude in Theorem \ref{thmvlow}. 
	Recall that $\beta$ is the 
		growth exponent of the three-dimensional LERW and $B_\ust$ indicates intrinsic balls in the 
		three-dimensional UST $\ust$ on a probability space $(\Omega,\mathcal{F},\prob)$. 
	
	\begin{prop}\label{VUpper}
		There exist $c_6,c_7>0$ such that for all $\lambda>0$ and $r\ge 1$, 
		\begin{equation}\label{vupper}
			\prob(|B_\ust(0,r)|\le \lambda^{-1}r^{3/\beta})\ge c_6\exp\{-c_7\lambda^{\beta/(3-\beta)}\log \lambda\}.
		\end{equation}
	\end{prop}
	
	\begin{remark}
		See \cite{ACHS20}*{Theorem 5.1} for an exponential upper bound for the probability in (\ref{vupper}). 
	\end{remark}
	
	Here we follow the idea of \cite{BCK21}, Section 3 and 4 again. 
	Recall that $\gamfty$ is the infinite LERW started at the origin as the first step of Wilson's algorithm. 
	For the result at (\ref{vupper}), we will take a similar approach to the previous section, 
		in which we construct the UST $\ust$ by Wilson's algorithm in a collection of small boxes.  
	We consider a rectangular prism of side-length $(2N-1)m$, $(2N-1)m$, $2Nm$ 
		as a collection of $2N(2N-1)^2$ small boxes of side-length $m$. 
	We let the origin be located at the center of one of the two boxes closest to the center of the 
		large rectangular prism. 
	Let $\pi$ be an sequence of the boxes that starts at the one containing the origin and 
		spirals outwards. 

	Now we describe an example of how to construct such a spiral inductively. 
	In the case of $N=1$, we start at the box containing the origin and then move to the another one. 
	Without loss of generality, we can let the latter cube be centered at $(0,0,m)$ and 
		$\pi$ go upwards. 
	In the case of $N=2$ (also see Figure 2), we first continue to move upwards to the box 
		centered at $(0,0,2m)$ and spiral outwards in the upper face of the cube. 
	Then we spiral down along the side of the cube. 
	Finally we spiral inwards the lower face of the cube and end up with the box centered at $(0,0,-m)$. 
	Suppose that we have constructed the spiral up to the $N$-th step. 
	If $N$ is odd, the last step ended up with the box centered at $(0,0,Nm)$, therefore we move to 
		the box centered at $(0,0,(N+1)m)$ and 
		continue in the same procedure as the previous case ($N=2$). 
	If $N$ is even, the last step ended up with the box centered at $(0,0,-(N-1)m)$, therefore we first 
		move to the box centered at $(0,0,-Nm)$ and spiral outwards the lower face of the cube, 
		spiral up along the side of the cube and then spiral inwards in the upper face of the cube. 

	\begin{figure}[htb]
		\centering
		\includegraphics[width=0.75\linewidth]{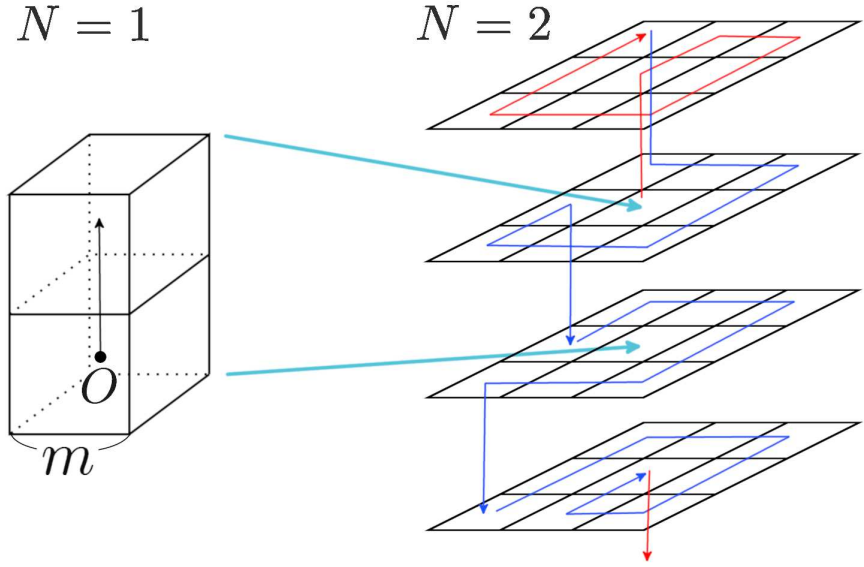}
		\caption{An example of three-dimensional spiral}
	\end{figure}
	
	\begin{remark}
		Note that the following argument can be applied to any spiral $\pi$ which step by step goes outwards 
			without getting close to the origin. 
	\end{remark}

	In order to prove Proposition \ref{VUpper}, we consider several events similar to 
		those which we considered in the previous section.
	Let $\{x_j\}_{j=1}^{2N(2N-1)^2}$ be a sequence of the center of the boxes in the $\pi$, \ie 
		$\pi=\{\Bx{j}\}_{j=1}^{2N(2N-1)^2}$, where $\Bx{j}=B_\infty(x_j,m/2)$. 
	Now we define some events for the SRW $S$ started at the origin. 
	By linear interpolation, we may assume that $S[0,\infty)$ is a continuous curve in $\rrr$. 
	\begin{defi}\label{def44}
	For $j=1,2,\cdots, 2N(2N-1)^2$, let $Q_j$ be the face of $\Bx{j}$ which is 
		the closest to $x_{j+1}$ and let $w_j\in \rrr$ be the center of $Q_j$. 
	Then we define $\tQ_j$ by 
	\[
		\tQ_j=Q_j\cap B_\infty(w_j,m/4),
	\]
	a subset of the face $Q_j$ which is not too close to its edges. 
	For $a>0,b\ge 0$, we define 
	\begin{align*}
		Q_j[-a,-b]&=\{y\in \Bx{j}\mathrel{:} b\le d_E(y,Q_j)\le a\},	\\
		Q_j(-a)&=\{y\in \Bx{j}\mathrel{:} d_E(y,Q_j)=a\},	\\
		R_j&=Q_j(-m/2)\cap (B(x_j,m/8)^c\cup B(x_j,m/10)).
	\end{align*}
	For a continuous surve in $\rrr$ and $a\in \mathbb{R}$, we define
	\[
		t_\lambda(Q_j(a))=\inf \{k\ge 0\mathrel{:}\lambda(k)\in Q_j(a)\}.
	\]
	\end{defi}
	The sets $\widetilde{Q}_j$, $Q_j[-a,-b]$, $Q_j(-a)$ and $R_j$ we defined above correspond to the idea of 
		$\widetilde{Q}(a_{j+1})$, $Q[a_{j+1}-a,a_{j+1}-b]$, $Q(a_{j+1}-a)$ and $R_j$, 
		which we defined in Definition \ref{defsec3} in Section 3, respectively. 

\begin{figure}[htbp]
	\begin{minipage}[b]{0.48\linewidth}
		\centering
		\includegraphics[width=1\linewidth]{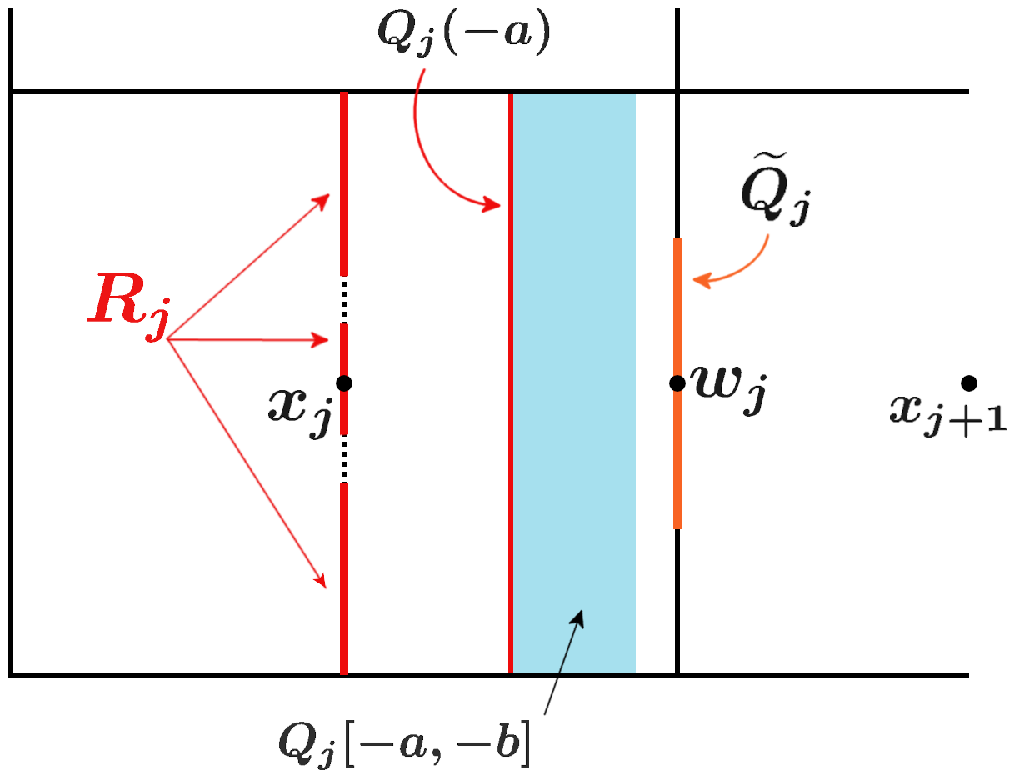}
		\caption{The sets defined in Definition \ref{def44}}
	\end{minipage}
	\hspace{0.01\linewidth}
	\begin{minipage}[b]{0.48\linewidth}
		\centering
		\includegraphics[width=1\linewidth]{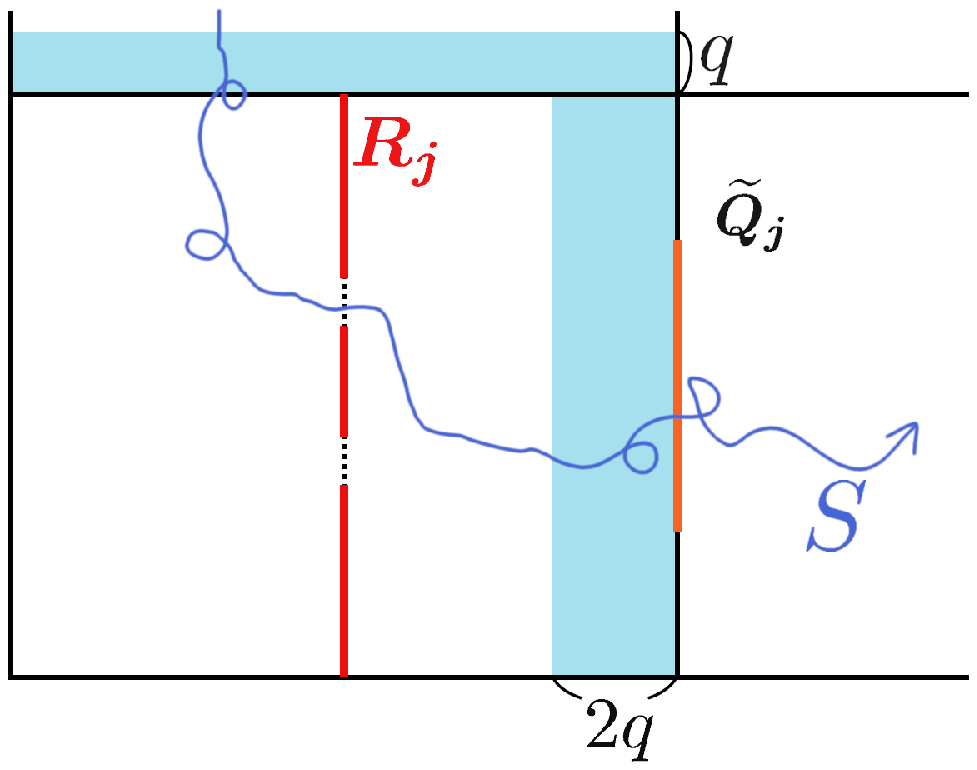}
		\caption{Definition of the event $A'_j$}
	\end{minipage}
	\end{figure}
	
	Now we define some events for the SRW $S$. Let $q=m/N$ and
	\begin{align}
	A'_1=\{t_S&(Q_1(0))<\infty, S\left(t_S(Q_1(0))\right)\in \tQ_1, S[t_S(Q_1(-q)),t_S(Q_1(0))]\cap Q_1(-2q)=\emptyset\},\notag	\\
	A'_j=\{t_S&(Q_{j-1}(0))<t_S(Q_j(0))<\infty, S\left(t_S(Q_j(0))\right)\in \tQ_j, \notag	\\
		&S[t_S(Q_{j-1}(0)),t_S(Q_j(0))]\subset Q_{j-1}[-q,0]\subset\Bx{j}\setminus R_j,	\notag	\\
		&S\left[t_S(Q_j(-q)),t_S(Q_j(0))\right]\subset Q_j[-2q,0]\},\ \ \ \ \ \ \ \ \ \ \mbox{for }j\ge 2. \label{A'j}
	\end{align}
	
	Note that the event $A_1$ (resp. $A_j$, $j\ge 2$) is measurable with respect to 
		$S[0,t_S(Q_j(0))]$ (resp. $S[t_S(Q_{j-1}(0)),t_S(Q_j(0))]$). 
		
	We set 
	\[
		G'_j=\bigcap_{k=1}^j A'_k.
	\]

	Now we define a cut time for $S$. 
	\begin{defi}
	Suppose that the event $A'_j$ defined in (\ref{A'j}) occurs. For each $j\ge 2$, we call $k$ is a 
		nice cut time in $\Bx{j}$ if it satisfies the four conditions in Definition \ref{cuttimedef} 
		with $t_S(a_j+b)$ replaced by $t_S(Q_j(b))$ for $b\in \mathbb{R}$. 
	
	If $k$ is a nice cut time in $\Bx{j}$, then we call $S(k)$ is a nice cut point in $\Bx{j}$. 
	\end{defi}
	We define events $B'_j$ by 
	\begin{equation}\label{B'j}
		B'_j=\{S \mbox{ has a nice cut point in }\Bx{j}\}, 
	\end{equation}
	for each $j\ge 2$. 
	Note that event $B'_j$ is measurable with respect to $S[t_S(Q_{j-1}(0)),t_S(Q_j(0))]$. 
	We define 
	\begin{equation*}
		H'_j=\bigcap_{k=2}^j B'_k.
	\end{equation*}
	\begin{figure}[htb]
		\centering
		\includegraphics[width=0.6\linewidth]{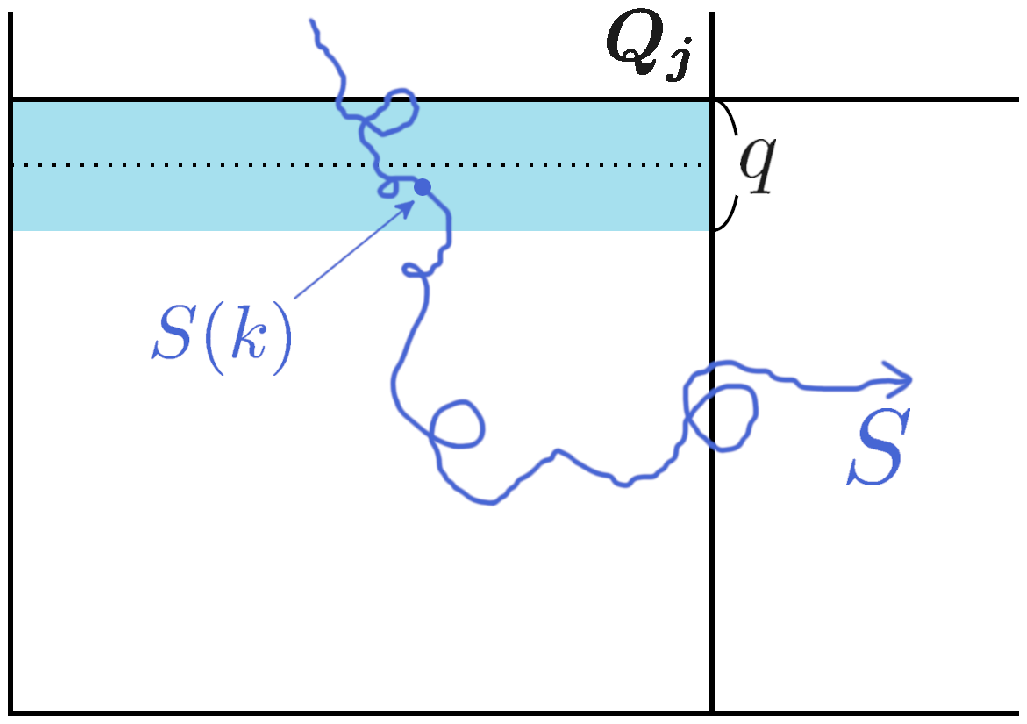}
		\caption{Example of nice cut point in $B_{x_j}$}
	\end{figure}
	
	The events $A'_j$ and $B'_j$ correspond to the idea of the events $A_j$ and $B_j$, which we defined in Section 3 
		to estimate upper fluctuation of the volume of the three-dimensional UST. 

	Now we consider the length and the hittability of the loop erasure of $S$. 
	Suppose that the event $G'_j\cap H'_j$ occurs and let $k_j$ be a nice cut time in $\Bx{j}$. 
	We set 
	\begin{align}
		\xi''_j&=\mathrm{LE}(S[0,t_S(Q_j(0))])~\mbox{ for }j\ge 1,	\notag	\\
		\lambda'_j&=\mathrm{LE}(S[k_j,t_S(Q_j(0))])~\mbox{ for }j\ge 2,	\notag
	\end{align}
	and let 
	\begin{align}
		s_j&=\inf\{k\ge 0\mathrel{:}\xi''_{j-1}\in S[t_S(Q_{j-1}(0)),t_S(Q_j(0))]\},	\notag	\\
		t_j&=\sup\{t_S(Q_{j-1}(0))\le k\le t_S(Q_j(0))\mathrel{:}S(k)=\xi''_{j-1}(s_j)\},	\notag	\\
		u_j&=\inf\{k\ge 0\mathrel{:} \lambda'_j(k)\in Q_j[-q,0]\},	\notag
	\end{align}
	for $j\ge 2$. Then we have 
	\[
		\xi''_j=\xi''_{j-1}[0,s_j]\oplus \mathrm{LE}(S[t_j,k_j])\oplus\mathrm{LE}(S[k_j,t_S(Q_j(0))])
			\supset \xi''_{j-1}[0,s_j]\cup \lambda'_j,
	\]
	and therefore,
	\[
		\xi_j[0,s_{j+1}]\supset \xi''_{j-1}[0,s_j]\cup \lambda'_j[0,u_j],
	\]
	on the event $G'_j\cap H'_j$. 
	\begin{figure}[htb]
		\centering
		\includegraphics[width=0.5\linewidth]{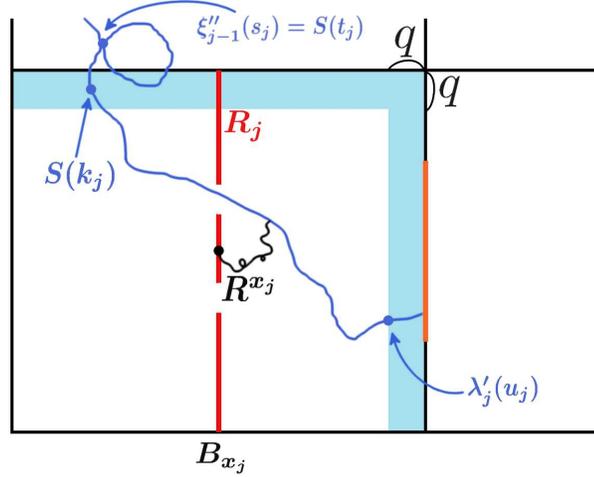}\label{eventEFpj}
		\caption{The sets and the points we consider in the events $E'_j$ and $F'_j$}
	\end{figure}
	Thus, in order to bound the length of $\eta_j$ from below, we need to estimate the length of $\lambda'_j[0,u_j]$.
	For $j\ge 2$ and $C>0$, we define the event $E'_j(C)$ by
	\begin{equation}\label{E'j}
		E'_j=E'_j(C)=\{\mathrm{len}(\lambda'_j[0,u_j])\ge Cm^\beta\}.
	\end{equation}
	Moreover, we define the event $F'_j(\eta)$ for the hittability of $\lambda'_j$ for $j\ge 2$ by
	\begin{equation}\label{F'j}
		F'_j=F'_j(\eta)=\{P^{x_j}(\lambda'_j\cap(R^{x_j}[0,T_{R^{x_j}}(2m/5)])\neq\emptyset)\ge \eta\},
	\end{equation}
	where $P^z$ indicates the law of $R^z$, the simple random walk started at $z$ independent 
		of $S$. 
				
	The next lemma is an analog of Lemma \ref{c-eta}, which gives a lower bound on the probability that 
		the events as defined above occur simultaneously. 
	\begin{lem}\label{lem_low}
		Let $P$ be the law of $S$. 
		There exist universal constants $0<\eta^*,c^*<\infty$ such that 
		\begin{equation*}
			P(A'_1)\ge c^*
		\end{equation*}
		and for all $j\ge 2$,
		\begin{equation}\label{4ev_low}
			\min_{x\in\tQ_j}P^x(A'_j\cap B'_j\cap E'_j(c^*)\cap F'_j(\eta^*))\ge c^*N^{-2}.
		\end{equation}
	\end{lem}	
	\begin{proof}
	Similarly to the proof of Lemma \ref{c-eta}, it sufficies to show that for a fixed constant $c$, 
		there exists a universal constant $c'$ such that $P^x(E'_2(c') \mid A'_2)\ge c$
		holds uniformly in $x\in\tQ_1$. 
	We consider a small box $B_\infty(x_2,m/6)$ of side-length $m/3$, 
		which is included in $\Bx{2}$. 
	Let $I$ be the number of points lying in both $\lambda'_2$ and $B_\infty(x_2,m/6)$. 
	Then by \cite{S18}*{Lemma 8.9}, there exists a universal constant $c'$ such that 
	\[
		P^x(I \ge c'm^\beta\mid A'_2)\ge c.
	\]
	Since $|\mathrm{len}(\lambda'_2[0,u_2])|\ge I$, we have that 
	\[
		P^x(|\lambda'_2[0,u_2]|\ge c'm^\beta \mid A'_2)\ge c,
	\]
	uniformly in $x\in \tQ_1$, which completes the proof.
	\end{proof}

	For $j=1,2,\cdots, L(l_N)$, let $M'_j(\lambda)$ be an event defined by 
	\begin{equation*}
		M'_j(\kappa)=\{B(x_j,\kappa^{-2}m)\subset B_{\ust^N}(x_j,\kappa^{-1}m^\beta\}.
	\end{equation*}
	Let $c_*$ be a constant which satisfies (\ref{4ev_low}). 
	By the same argument as Lemma \ref{balls}, there exists some constant $\kappa_*>0$ 
		depending only on $c_*$ such that 
	\begin{equation}\label{kappa*}
		\prob\left(M'_j(\kappa_*)) \mid A'_j\cap B'_j\cap E'_j(c_*)\cap F'_j(\eta_*) \right) \ge c_*.
	\end{equation}
	We set $I'_j=\bigcap_{k=2}^j I'_k(\kappa_*)$ for $\kappa_*$ defined in (\ref{kappa*}). 

	Let $J'_j=\bigcap_{k=2}^j E'_k(c^*)$ and $K'_j=\bigcap_{k=2}^j F'_k(\eta^*)$
	for $c^*$ and $\eta^*$ defined in Lemma \ref{lem_low}. 
	We define
	\begin{align}
		U'_N=&\left\{S(T_S(R(N)))\in\{(y^1,y^2,y^3)\in\rrr\mathrel{:} y^1\ge \frac{4}{5}R(N)\},\right. \notag	\\
		&\ \ \left.S[t_S(Q_{4N(4N-1)^2}),T_S(R(N))]\cap B\left(0,3\left(N-\frac{1}{2}\right)\right)=\emptyset\right\},	\label{U'N}
	\end{align}
	where $R(N)=\exp\{2e^{RN^2}+1\}$ and set 
	\begin{equation}\label{defofB^N}
		B^N=G'_{2N(2N-1)^2}\cap H'_{2N(2N-1)^2}\cap I'_{2N(2N-1)^2}\cap J'_{2N(2N-1)^2}\cap K'_{2N(2N-1)^2}\cap U'_N.
	\end{equation}
	\begin{figure}[htb]
		\centering
		\includegraphics[width=0.8\linewidth]{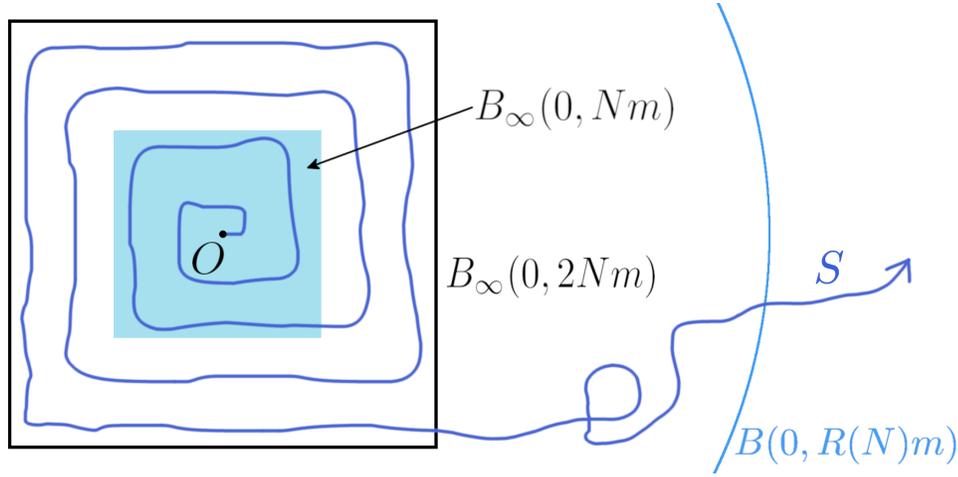}
		\caption{Definition of the event $U'_{N}$}
	\end{figure}

	Then by the same argument as Lemma \ref{x1axis}, we obtain the following lemma. 
	\begin{lem}
		There exists a universal constant $c_8>0$ such that 
		\begin{equation}\label{defBtilde}
			P(B^N)\ge c_8^{-1}\exp \{-c_8N^3(\log N)\}.
		\end{equation}	
	\end{lem}
	
	On the event $B^N$, the first LERW $\gamfty$ in Wilson's algorithm moves through the 
		spiral $\pi$ and its length up to the $j$-th box is bounded below by $Cjm^\beta$. 
	
	\begin{lem}\label{leminclu}
	There exist universal constants $c,c',C$ such that for all $m\ge 1$,
	\begin{equation}\label{inclu}
		\prob\left(B_\ust\left(0,cN^3 m^\beta\right)\subset B_\infty\left(0,\frac{2}{3}Nm)\right)\right) \ge C\exp\{-c'N^3(\log N)\}.
	\end{equation}
	\end{lem}
	
	Once we prove the above lemma, we will obtain Proposition \ref{VUpper} as follows. 
	\begin{proof}[Proof of Propositon \ref{VUpper}]
	It follows from (\ref{inclu}) that 
	\[
		\prob\left(\left|B_\ust\left(0,cN^3 m^\beta\right)\right|\le \frac{8}{27}N^3m^3\right)\ge C\exp\{-c'N^3(\log N)\}.
	\]
	Reparameterizing $r=cN^3m^\beta$ and 
		$\lambda=\frac{27}{8}c^{3/\beta}N^{9/\beta-3}$ yields the result at (\ref{vupper}).
	\end{proof}

	\begin{proof}[Proof of Lemma \ref{leminclu}]
	We may assume that $m$ is sufficiently large for the same reason as \cite{ACHS20}*{Proposition 4.1}. 
	
	We first take a sequence of subsets of $\zzz$ including the boundary of $B(0,(2N-1)m/3)$. 
	For each $k\ge 1$, let $\vep_k=N^{-4/3}2^{-k-10}$, $\eta_k=(15k)^{-1}$ and 
	\[
		A_k=B_\infty\left(0,\left(\frac{2}{3}+\eta_k\right)(2N-1)\frac{m}{2}\right)\setminus B_\infty\left(0,\left(\frac{2}{3}-\eta_k\right)(2N-1)\frac{m}{2}\right).
	\]
	We let $D_k\subset \zzz$ be a subset of $A_k$ such that 
		$A_k\subset \bigcup_{z\in D_k} B(z,(2N-1)m\vep_k)$ and we suppose that 
		$|D_k|\le C\vep_k^{-3}$. 
	Write $k_0$ for the smallest integer satisfying $(2N-1)m\vep_{k_0}<1$. 
	Note that for sufficiently large m, both the inner and outer boundary of 
		$B_\infty(0,\frac{1}{3}(2N-1)m)$ are contained in $A_{k_0}$. 
	Moreover, we have $\partial_i B_\infty(0,\frac{1}{3}(2N-1)m)\subset D_{k_0}$ by the definition 
		of $k_0$. 
	
	We begin with performing Wilson's algorithm in $\bigcup_{j=1}^{2N(2N-1)^2} \Bx{j}$. 
	Let $\ust^N_0$ be the subtree of the UST constructed in the event $B^N$. 		
	Then we follow the same construction of the sequence of subtrees 
		$\ust^N_1,\ust^N_2,\cdots, \ust^N_{k_0}$ 
	as (i)-(iv) in the proof of Lemma \ref{balls} and end up with the subtree $\ust^N$, 
		which contains all points in $\bigcup_{j=1}^{2N(2N-1)^2} B(x_j,m/2)$.

	Next we consider the hittability of branches in the constructed subtree. 
	We take $z\in D_1$, then there exists $j\in[1,2,\cdots,2N(2N-1)^2]$ such that $z\in\Bx{j}$. 
	Let $\tau_{B_j}$ be the first hitting time of $B(x_j,\lambda^{-2}m)$ and $\sigma_{Q_j}$ 
		be the first exiting time of $B_\infty(x_j,3m/2)$ by $R^z$, a simple random walk started 
		at $z$ and independent of $S$. 
	Then by \cite{L96}*{Proposition 1.5.10}, there exists some $p>0$ such that for all $j\ge 1$ and $z\in \Bx{j}$ 
	\begin{equation}\label{sigmaQ} 
		P^z(\sigma_{Q_j}<\tau_{B_j})\le p,
	\end{equation}
	holds. 
	We define the event $L(1,z)$ by
	\[
		L(1,z)=\{\mbox{$\gamma_{\ust^N}(z,\infty)$ exits the cube $B_\infty(z,Nm/100)$ before hitting the subtree $\ust^N_0$}\},
	 \]
	for $z\in D_1$ and suppose that $L(1,z)$ occurs, 
		where $\gamma_{\ust^N}(y,\infty)$ stands for the unique infinite path of the subtree $\ust^N$
		started at $y$.
	Then the event $\{\sigma_{Q_j}<\tau_{B_j}\}$ occurs and by (\ref{sigmaQ}) 
		its probability is smaller than $p$. 
	By iteration, the number of boxes of side-length $3m/2$ that $R^z$ exist 
		before hitting $\ust^N_0$ is larger than $N/200$. 
	Hence, by the strong Markov property, it follows that 
	\[
		\prob(L(1,z))\le p^{N/200},
	\]
	for all $z\in D_1$. Since $|D_1|\le C\vep_1^{-3}$, we have that 
	\[
		\prob\left(\bigcap_{z\in D_1} L(1,z)^c \right)\ge 1-N^4 p^{N/200}.
	\]

	We next define the event that guarantees the hittability of branches of $\ust^N_k$ starting 
		at $z\in D_k$. 
	For $k\ge 1$ and $x\in D_k$, let 	
	\begin{align}
		I'&(k,x,\zeta)	\notag	\\
		&=\left\{P_R^y\left(R\left[0,T_{R^y}(y,(2N-1)m\vep_k^{1/2})\right]\cap (\ust^N_0\cup \gamma_{\ust^N}(x,\ust^N_0))\right)\le \vep_k^\zeta \mbox{ for all }y\in B(x,(2N-1)m\vep_k)\right\}.
	\end{align}
	and $I'(k,\zeta)=\bigcap_{x\in D_k}I'(k,x,\zeta)$. 
	Again by \cite{SaSh18}*{Lemma 3.2}, there exist universal constants 
		$\zeta_3>0$ and $C<\infty$ such that for all $k\ge 1$, $m\ge 1$, and 
		$x\in D_k$,
	\[
		\prob_{\ust^N}(I'(k,x,\zeta_3)^c)\le C\vep_k^5.
	\]
	Combining this with $|D_k|\le C\vep_k^{-3}$ yields that 
	\[
		\prob_{\ust^N}(I'(k,\zeta_3)^c)\le C\vep_k^2.
	\]	
	
	We finally define an event $L(k,z)$ for $k\ge 2$ and $z\in D_k$ by
	\[
		L(k,z)=\{\mbox{$\gamma_{\ust^N}(z,\infty)$ exits $B(z,(2N-1)m\vep_{k-1}^{1/3})$ before hitting $\ust^N_{k-1}$}\},
	\]
	and set $M_1\coloneqq B^N\cap (\bigcap_{z\in D_1} L(1,z)^c)\cap I'(1,\zeta_3)$ 
		and $M_k\coloneqq M_{k-1}\cap (\bigcap_{z\in D_1} L(1,z)^c) \cap I'(k,\zeta_3)$ 
		inductively for $k\ge 2$.
	Suppose that the event $M_{k-1}$ occurs. 
	The number of balls of radius $(2N-1)m\vep_{k-1}^{1/2}$ that $R^z$ 
		exits before hitting $\ust^N_{k-1}$ is larger than $\vep_{k-1}^{-1/6}$. 
	By the strong Markov property, it holds that   
	\begin{equation*}
		P^z(R^z\mbox{ exits }B(z,(2N-1)m\vep_1^{1/3})\mbox{ before it hits }\ust^N_{k-1})\le \vep_{k-1}^{c\zeta_3\vep_{k-1}^{-1/6}},
	\end{equation*}
	for some universal constant $c>0$. 
	Since $|D_k|\le C\vep_k^{-3}$, we have that
	\begin{equation}\label{epseps}
		\prob\left(\bigcap_{z\in D_k} L(k,z)^c\relmiddle| M_{k-1}\right)\ge 1-C\vep_k^{-3}\vep_{k-1}^{c\zeta_3\vep_{k-1}^{1/6}},
	\end{equation}
	
	It follows from the argument above that 
	\[
		\prob(M_1\mid B^N)\ge 1-N^{-8/3},
	\]
	and
	\[
		\prob(M_k\mid M_{k-1}) \ge 1-C\vep_k^2.
	\]
	Hence we can conclude that 
	\begin{equation*}
		\prob(M_{k_0}\mid B^N)=\prob(M_1\mid B^N)\prod_{k=2}^{k_0} \prob(M_k\mid M_{k-1}) \ge (1-N^{-8/3})\prod_{k=1}^\infty (1-C\vep_k^2)\ge 1-CN^{-8/3}.
	\end{equation*}

	On the event $M_{k_0}$, for all $y\in \ust^N$,
	\begin{equation*}
		d_E(y,D_1)\le \frac{Nm}{100}+\sum_{k=2}^{k_0} (2N-1)m\vep_k^{1/6} \le \frac{Nm}{100}+(2N-1)m\sum_{k=2}^{k_0}\vep_k^{1/6} \le \frac{Nm}{50},
	\end{equation*}
	\ie $\gamma_\ust(y,0)$ hits the subtree $\ust^N_0$ before it exits the ball centered at $y$ 
		of radius $Nm/50$. 
	Hence we have
	\[
		d_E(0,y)\ge d_E(0,D_1)-\frac{Nm}{50}\ge \frac{11}{20}(2N-1)m.
	\]
	Since $\partial_i B_\infty(0,\frac{1}{3}(2N-1)m)\subset D_{k_0}$, 
		$\gamma_{\ust}(z,0)$ hits $\ust^N_{k_0}$ before 
		entering $B_\infty(0,\frac{11}{20}(2N-1)m)$ for all $z\in B_\infty(x,\frac{1}{3}(2N-1)m)$. 
	By the definition of the spiral $\pi$, it follows that on the event $M_{k_0}$, 
	\[
		z\in B_\infty\left(0,\frac{1}{3}(2N-1)m\right)^c \Longrightarrow d_{\ust}(0,z)\ge CN\left(N-\frac{1}{2}\right)^2m^\beta
	\]
	for some universal constant $c$, \ie 
		$B_\ust(0,cN(N-\frac{1}{2})^2m^\beta)\subset B_\infty(0,\frac{1}{3}(2N-1)m)$.
	By taking $N$ sufficiently large, we obtain (\ref{inclu}) for some universal constant $c'$. 
	\end{proof}
	
	\begin{theo}\label{thmvlow}
		$\prob$-a.s., 
		\begin{equation}\label{vlowfluc}
			\liminf_{r\to\infty}(\log\log r)^{(3-\beta)/\beta}r^{-3/\beta}|B_\ust(0,r)|=0.
		\end{equation}		
	\end{theo}
	\begin{proof}
	Similarly to the proof of upper volume fluctuation in Theorem \ref{thmvup}, 
		we begin with defining a sequence of scales by 
	\[
		D_i=e^{i^2},\ m_i=D_i/\vep(\log i)^{1/3}.
	\]
	Let $\gamfty$ be the infinite LERW started at the origin and 
		$(S^z)_{z\in\zzz}$ be the family of independent SRW which is 
		also independent of $\gamfty$. 
	Then by the same argument as obtaining (\ref{inc2}) and (\ref{inc3}), 
	\begin{align}
		\ust_i\subset B_\infty(0,i^{2M}D_i)&\subset B_\infty(0,m_{i+1}/2),	\\
		d_\ust(0,z)\le i^MD_i^\beta&\le m_{i+1}^\beta\ \ \ \mbox{for all}\ \ \ z\in\ust_i,
	\end{align}
	holds for large $i$, almost-surely. 
	
	We define an event $H(i)$ by 
	\[
		H(i)=\bigcap_{j=1}^i B^{\vep(\log j)^{1/3}},
	\]
	where the event $B^N$ is as defined in (\ref{defofB^N}) and at each stage $i$ we rescale by 
		$m=m_i$ and $N=\vep(\log i)^{1/3}$. 
	Then by (\ref{indp_const}), \ie the ``independence up to constant'' of $\gamfty$, 
		there exists a universal constant $C>0$ such that for any $i$, 
	\begin{align*}
		\prob(H(i+1)\mid \mathcal{F}_i)&\ge \prob(H(i+1)\mid \mathcal{F}_i)\mathbf{1}_{H(i)}	\\
			&\ge C\prob\left(B^{\vep(\log (i+1))^{1/3}}\right)	\\
			&\ge Ci^{-c\vep^3},
	\end{align*}
	where the last inequality follows from (\ref{defBtilde}). 
	Note that on $H(i)$ we have that 
	\[
		|B_\ust(0,cD_i^3m_i^{\beta-3})|\le \frac{8}{27}D_i^3.
	\]
	Finally, the reparameterization $r_i=cD_i^3m_i^{\beta-3}$ yields the result. 
	\end{proof}




\subsection{Estimates for effective resistance}

	In order to demonstrate upper heat kernel fluctuations, we need to estimate the effective resistance of balls in 
		the three-dimensional uniform spanning tree. 
	See (\ref{defofreff}) for the definition of effective resistanve. 
		
	Suppose that the event $B^N$ defined in (\ref{defBtilde}) occurs. 
	We give estimates for bounds of the volume and a lower bound of the effective resistance of UST 
		on the event $B^N$. 
	
	\begin{prop}\label{vvreff}	
		There exists some universal constant $0<c_9,c_{10},c,C<\infty$ and the event 
			$\widetilde{M}^N$ with $\prob(\widetilde{M}^N)\ge c_9\exp\{-c_{10}N^3\log N\}$ 
			such that on $\widetilde{M}^N$, the followings hold\red{:}
		\begin{align}
			&|B_\ust(0,N^3m^\beta)|\le \frac{1}{4}N^3m^3	\notag, \\
			&|B_\ust(0,CN^3m^\beta/(\log N)^{300})|\ge cN^3m^3/(\log N)^{300}	\notag, \\
			&\Reff{x}{B_\ust(0,N^3m^\beta)}\ge cN^3m^\beta/(\log N)^{300}\mbox{ for all }x\in B_\ust(0,cN^3m^\beta/(\log N)^{300}).	\label{reffbound}
		\end{align}
	\end{prop}
	We will apply \cite{KuMi08}*{Proposition 3.2} to bound the heat kernel on $\ust$ from above. 
	In order to do so, we need to estimate (i) upper bound of the volume of an intrinsic ball 
		$B_\ust(0,N^3m^\beta)$, (ii) lower bound of the volume of a smaller ball 
		$B_\ust(0,\vep N^3m^\beta)$ where $\vep$ is some small constant and 
	(iii) lower bound of the effective resistance between $x\in B_\ust(0,\vep N^3m^\beta)$ and 
		the boundary of the ball $B_\ust(0,N^3m^\beta)$, which correspond to the three 
			inequalities above. 

	\begin{figure}[htb]
		\centering
		\includegraphics[width=0.6\linewidth]{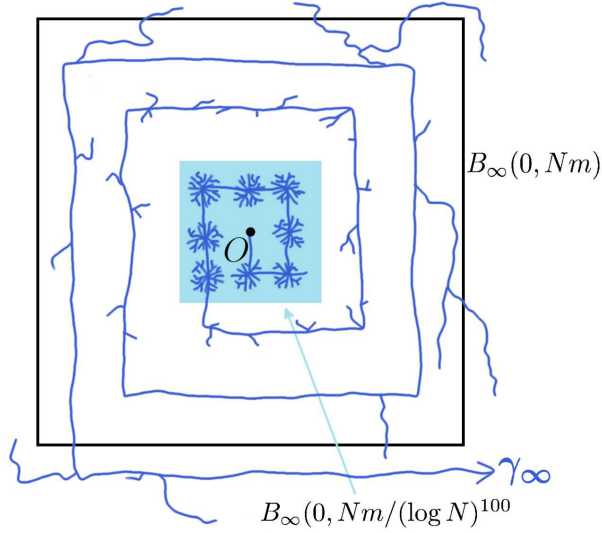}\label{resistance_bound}
		\caption{The unlike event which we consider in Proposition \ref{vvreff}. In order to bound the effective resistance from below, we need to consider two properties: (i) paths started at a point in $B_\ust(0,N^3m^\beta)^c$ do not enter a smaller ball $B_\ust(0,N^3m^\beta/(\log N)^{300})$ and (ii) paths branching from $\gamfty$ at a point close to the origin have limited length.}
	\end{figure}
	
	\begin{proof}
		We set $l_N=N/(\log N)^{100}$ and $L(k)=2k(2k-1)^2$. 
		Let $S$ be the SRW started at the origin and recall that events $A'_j$, $B'_j$, 
			$E'_j$, $F'_j$ and $U'_N$ are as defined in (\ref{A'j}), (\ref{B'j}), (\ref{E'j}), (\ref{F'j}) and 
			(\ref{U'N}), respectively. 
		We define events $B^N_{(1)}$ and $B^N_{(2)}$ by
		\begin{align*}
			B^N_{(1)}&=\bigcap_{j=1}^{L(l_N)}\left(A'_j\cap B'_j\cap E'_j\cap F'_j\cap \{B(x_j,\lambda^{-2}m)\subset B_{\ust}(x_j,\lambda^{-1}m^\beta)\}\right),	\\
			B^N_{(2)}&=\left\{\bigcap_{j=L(i_N)+1}^{L(N)}\left(A'_j\cap B'_j\cap E'_j\cap F'_j\cap \{B(x_j,\lambda^{-2}m)\subset B_{\ust}(x_j,\lambda^{-1}m^\beta)\}\right)\right\}\cap U'_N.
		\end{align*}
		It follows from the same argument as Lemma \ref{leminclu} that 
		\begin{align}
			\prob(B^N_{(1)})&\ge C\exp\{-cN^3(\log N)^{-299}\},\label{PBN(1)} \\ 
			\prob(B^N_{(2)})&\ge C\exp\{-cN^3(\log N)\},
		\end{align}
		for some constants $C,c>0$. 
		Note that $(\log N)^{-299}$ in the lower bound of $\prob(B_{(1)}^N)$ is the 
			result of the number of small boxes $l_N$. 
		We need this term in order to avoid the competition between $\prob(B_{(1)}^N)$ and 
			a bound for the probability that UST paths branching from $\gamfty$ near the origin 
			do not have large length.
		
		Suppose that the event $B^N_{(1)}\cap B^N_{(2)}$ occurs. 
		Then the occurence of $\bigcap_{j=1}^{L(N)} A'_j$ guarantees that the part of $S$ 
			after exiting the $L(l_N)$-th box $\Bx{{L(l_N)}}$ does not go back into 
			$\bigcap_{j=1}^{L(l_N)-1}\Bx{j}$. 
		In particular, $\mathrm{LE}(S[0,t_S(Q_{L(l_N)}(0))])$ and $\mathrm{LE}(S[0,t_S(Q_{L(N)}(0))])$ restricted to $\bigcap_{j=1}^{L(l_N/2)}\Bx{j}$ are exactly the same on $B^N_{(1)}\cap B^N_{(2)}$, 
			where $t_\lambda(Q_j(a))$ is as defined in Definition \ref{def44}. 
		
		For each $k\ge 1$, let $\vep_k=N^{-4/3}2^{-10-k}$, $\eta_k=(30k)^{-1}$ and 
		\[
			A_k=B\left(0,\left(\frac{1}{3}+\eta_k \right)l_N m\right)\setminus B\left(0,\left(\frac{1}{3}-\eta_k \right)l_N m\right).
		\]
		Write $k_0$ for the smallest integer satisfying $l_N m\vep_k<1$. 
		Now we take a ``$\vep_k$-net'' of $A_k$, \ie let $D_k$ be a set of lattice points in $A_k$ such that 
			$\bigcup_{z\in D_k} B\left(z,l_N m\vep_k\right)$ with $|D_k|\le C\vep_k^{-3}$. 
		
		Now we perform Wilson's algorithm in $\bigcup_{j=1}^{L(N)}\Bx{j}$ to 
			obtain the subtree $\ust^{L(N)}_{k_0}$ in the same procedure as we did in Lemma \ref{leminclu}. 
		Note that $\ust_0^{L(N)}$ is the union of the infinite LERW $\gamfty$ started at the origin 
		and balls $B_\ust(x_j,\lambda^{-1}m^\beta)$ constructed in the event $I'_{2N(2N-1)^2}$. 
		For $z\in\ust_{k_0}^{L(N)}$, let $\gamma_{\ust^{L(N)}}(x,\infty)$ be the unique 
			infinite path in $\ust_{k_0}^{L(N)}$ starting at $z$. 
		We set $C_p=5\log N/\log(1/p)$ where $p$ be as defined in (\ref{sigmaQ}). 
		For $z\in D_1$, we define the event $\widetilde{L}(1,z)$ by 
		\[
			\widetilde{L}(1,z)=\left\{\mathrm{LE}(S^z[0,\tau(\ust_0^{L(N)})]) \not\subset B\left(z,C_p m\right)\right\}\cup \left\{|\gamma_{\ust^{L(N)}}(z,\ust_0^{L(N)})|\ge \frac{N^3}{(\log N)^{10}}m^\beta\right\}.
		\]
		Then
		\begin{align}
			\prob&\left(\widetilde{L}(1,z) \relmiddle| B^N_{(1)}\cap B^N_{(2)}\right)	\notag	\\	
			&\le \prob\left(|\mathrm{LE}(S^z[0,\tau(\ust_0^{L(N)})])|\ge \frac{N^3}{(\log N)^{10}}m^\beta,\ \mathrm{LE}(S^z[0,\tau(\ust_0^{L(N)})]) \subset B\left(z,C_p m\right) \relmiddle| B^N_{(1)}\cap B^N_{(2)}\right)	\notag	\\
			&+\prob\left(\mathrm{LE}(S^z[0,\tau(\ust_0^{L(N)})]) \not\subset B\left(z,C_p m\right) \relmiddle| B^N_{(1)}\cap B^N_{(2)}\right)	\label{tildeL}.
		\end{align}
		Since both of the events we consider in the right-hand side of (\ref{tildeL}) are independent of 
			$B^N_{(2)}$, we can omit the condition on $B^N_{(2)}$ from the conditional 
			probabilities in the right-hand side. 
	For the first term, we stop conditioning on $B^N_{(1)}$ and consider 
			$\gamma_{\ust^{L(N)}}(z,\infty)$ as the infinite LERW started at $z$. 		
	Then it follows from \cite{S18}*{Theorem 1.4} and \cite{LS19}*{Corollary 1.3} that for some universal constants $C>0$ and $c>0$,
		\begin{align*}
			P^z&\left(|\mathrm{LE}(S^z[0,\tau(\ust_0^{L(N)})])|\ge \frac{N^3}{(\log N)^{10}}m^\beta,\ \mathrm{LE}(S^z[0,\tau(\ust_0^{L(N)})]) \subset B\left(z,C_p m\right) \relmiddle| B^N_{(1)}\cap B^N_{(2)}\right)	\\
			&\le \frac{P^z\left(|\mathrm{LE}(S^z[0,\tau(\ust_0^{L(N)})])|\ge N^3 m^\beta/(\log N)^{10},\ \mathrm{LE}(S^z[0,\tau(\ust_0^{L(N)})]) \subset B\left(z,C_p m\right)\right)}{\prob\left(B^N_{(1)}\right)}	\\
			&\le C\exp\{-cN^3(\log N)^{-10-\beta}\},
		\end{align*}
		where the last inequality follows from (\ref{PBN(1)}). 
		On the other hand, by the independence, the strong Markov property and the hittability 
			of $\ust_0^{L(N)}$, 
		\begin{align*}
			P^z&\left(\mathrm{LE}(S^z[0,\tau(\ust_0^{L(N)})]) \not\subset B\left(z,C_p m\right) \relmiddle| B^N_{(1)}\cap B^N_{(2)}\right)	\\
			&\le p^{-C_p}=N^{-5}.
		\end{align*}
		Thus, we obtain that 
		\begin{equation}\label{tildeLest}
			\prob\left(\widetilde{L}(1,z) \relmiddle| B^N_{(1)}\cap B^N_{(2)} \right) \le N^{-5}.
		\end{equation}

		Next we consider events which guarantees the hittability of each branches in $\ust^{L(N)}_k$. 
		For $\zeta>0$, we define an event $\widetilde{I}(k,x,\zeta)$ by 
		\begin{align}
			\widetilde{I}(k,x,\zeta&)	\notag	\\
				=\Bigl\{&P_R^y\left(R\left[0,T_{R^y}(y,l_N m\vep_k^{1/2})\right]\cap (\ust^{L(N)}_0\cup \gamma_{\ust^{L(N)}}(x,\ust^{L(N)}_0))=\emptyset\right)\le \vep_k^\zeta \notag	\\
				&\mbox{ for all }y\in B(x,l_N m\vep_k)\Bigr\}.
		\end{align}
		and let $\widetilde{I}(k,\zeta)=\bigcap_{x\in D_k} \widetilde{I}(k,x,\zeta)$. 
		Then by \cite{SaSh18}*{Lemma 3.2}, there exist universal constants $\zeta_4>0$ 
			and $C<\infty$ such that for all $k\ge 1$, $m\ge 1$, and $x\in D_k$,
		\[
			\prob(\widetilde{I}(k,x,\zeta_4))\ge1- C\vep_k^5,
		\]
		which combined with $|D_k|\le C\vep_k^{-3}$ yields 
		\begin{equation}\label{tildeI}
			\prob(\widetilde{I}(k,\zeta_4))\ge 1-C\vep_k^2.
		\end{equation}
		
		Finally, let 
		\begin{align}
			\widetilde{L}(k,z)=&\left\{\gamma_{\ust^{L(N)}}(z,\infty)\mbox{ exits }B\left(z,l_N m\vep_{k-1}^{1/3}\right)\mbox{ before hitting }\ust^{L(N)}_{k-1}\right\}	\notag	\\
			&\cup\left\{\left|\gamma_{\ust^{L(N)}}(z,\ust^{L(N)}_{k-1})\right|\ge \vep_k^{1/5}\frac{N^3}{(\log N)^{10}}m^\beta\right\},
		\end{align}
		be the event for the length of the attached branch for each $k\ge 2$ and $z\in D_k$. 
		We set $\widetilde{M}_1=B^N_{(1)}\cap B^N_{(2)}\cap \widetilde{I}(1,\zeta_4)\cap (\bigcap_{z\in D_1}\widetilde{L}(1,z)^c)$ 
		and $\widetilde{M}_k=\widetilde{M}_{k-1}\cap \widetilde{I}(k,\zeta_4)\cap (\bigcap_{z\in D_k}\widetilde{L}(k,z)^c)$ inductively for $k=2,3,\cdots,k_0$. 
		
		Since $|D_1|\le CN^4$, it follows from (\ref{tildeLest}) that 
		\[
			\prob(\widetilde{M}_1)\ge (1-N^{-1})\exp\{-cN^3(\log N)\}.
		\]
		By applying the argument for (\ref{epseps}) again, 
		we have that , 
		\begin{align*}
			\prob&\left(\gamma_{\ust^{L(N)}}(z,\infty)\mbox{ exits }B\left(z,\frac{N}{(\log N)^{100}}m\vep_{k-1}^{1/3}\right)\mbox{ before hitting }\ust^{L(N)}_{k-1}
			\relmiddle| \widetilde{M}_{k-1}\right)	\\
			&\le 
			C\vep_{k-1}^{-3}\vep_{k-1}^{c\zeta_4\vep_{k-1}^{-1/6}}. 
		\end{align*}
		Again we stop conditioning on $\widetilde{M}_1$ and consider 
			$\gamma_{\ust^{L(N)}}(z,\infty)$ as the infinite LERW started at $z$. 
		By \cite{S18}*{Theorem 1.4} and \cite{LS19}*{Corollary 1.3}, 
		\begin{align*}
			\prob(\widetilde{L}&(2,z)\mid \widetilde{M}_{1})	\\
			\le&\prob\biggl(\left\{\gamma_{\ust^{L(N)}}(z,\ust^{L(N)}_1)\subset B\left(z,
			l_N m\vep_{1}^{1/3}\right)\right\} \cap \left\{\left|\gamma_{\ust^{L(N)}}(z,\ust^{L(N)}_{1})\right|\ge \vep_1^{1/5}\frac{N^3}{(\log N)^{10}}m^\beta\right\} \mid \widetilde{M}_{1}\biggr)	\\
			&+C\vep_{1}^{-3}\vep_{1}^{c\zeta_4\vep_{1}^{-1/6}}	\\	
			\le &\frac{\exp\{-\vep_2^{-c\beta/3+1/4}\cdot\vep_2^{-1/20}\cdot N^{3-\beta}(\log N)^{99}\}}{\prob(\widetilde{M}_1)}+C\vep_{1}^{-3}\vep_{1}^{c\zeta_4\vep_{1}^{-1/6}}	\\
			\le &\exp\{-\vep_2^{-1/12}N^3(\log N)^{99}\}.
		\end{align*}
		Combining this with (\ref{tildeI}), we obtain
		\[
			\prob(\widetilde{M}_2\mid \widetilde{M}_{1})\ge 1-C\vep_2^2,
		\]
		and by iteration, we have that $\prob(\widetilde{M}_k\mid \widetilde{M}_{k-1})\ge 1-C\vep_k^2$.
		Hence we can conclude that 
		\begin{align*}
			\prob(\widetilde{M}_{k_0})&=\prob(\widetilde{M}_1)\prod_{k=2}^{k_0}\prob(\widetilde{M}_k\mid \widetilde{M}_{k-1})	\\
			&\ge (1-N^{-1})\exp\{-cN^3\log N\}\prod_{k=1}^\infty(1-C\vep_k^2)	\\
			&\ge C\exp\{-cN^3\log N\},
		\end{align*}
		for some universal constant $C$,$c>0$. 
		On the event $\widetilde{M}_{k_0}$, we have that there exists some universal constant $C>0$ such that 
			for all $z\in \partial_i B(0,Nm/3(\log N)^{100})$, 
			$d_\ust(0,z)\le CN^3m^\beta/(\log N)^{10}$ holds. 
		
		Once we construct $\ust^{L(N)}_{k_0}$, we proceed with Wilson's algorithm. 
		This time, in the same argument as for (\ref{filll}), we have that conditioned on $\widetilde{M}_{k_0}$, 
			with conditional probability larger than some universal constant $c>0$, 
			$d_\ust(0,z)\le CN^3m^\beta/(\log N)^{10}$ holds for all $z\in B(0,Nm/3(\log N)^{100})$. 
		
		Finally, we consider ``$\vep_k$-net'' of annuli around the boundary of $B(0,Nm)$ and repeat the 
			argument of Lemma $\ref{leminclu}$. 
		Then we have that there exist universal constants $0<C,C',c,c'<\infty$ and an event $\widetilde{M}^N$ with 
			$\prob(\widetilde{M}^N)\ge C\exp\{-cN^3\log N\}$, such that the following statements hold: 
		\begin{enumerate}[(1)]
		\setlength{\leftskip}{2em}
			\item $|B_\ust(0,CN^3m^\beta/(\log N)^{300})|\ge cN^3m^3/(\log N)^{300}$,
			\item $\gamma_\ust(z,\gamfty)\not\subset B(0,c_2 l_N m)$  for all $z\in B(0,l_N m/3$, 
			\item $d_\ust(0,z)\le C_2N^3m^\beta/(\log N)^{10}$ for all $z\in B(0,l_N m/3$,
			\item $|B_\ust(0,N^3m^\beta)|\le \frac{1}{4}N^3m^3$.
		\end{enumerate}
		We bound the effective resistance $\Reff{0}{B_\ust(N^3m^\beta)}$ from below. 
		By (2), points outside the Euclidean ball $B(0,l_Nm/3)$ is connected to the 
			spiral $\gamfty$ outside the smaller ball $B(0,c_2l_N m)$. 
		Combining this with (3), all the paths on $\ust$ connecting the origin and 
			$B_\ust(0,C_2N^3m^\beta/(\log N)^{10})$ contains the part of $\gamfty$ 
			inside the $B(0,c_2 l_N m)$. 
		By the series law of effective resistance (see for example \cite{LP16}, section 2.3), 
			we have that $\Reff{0}{B_\ust(0,N^3m^\beta)^c}\ge c_3N^3m^\beta/(\log N)^{300}$ 
			for some universal constant $c_3>0$.
		Thus we have $\Reff{x}{B_\ust(0,N^3m^\beta)^c}\ge c_3N^3m^\beta/2(\log N)^{300}$ 
			for all $x\in B(0,c_3N^3m^\beta/2(\log N)^{300})$ , which completes the proof. 		
	\end{proof}
	

\section{Heat kernel fluctuations}



	In this section, we will show Theorem \ref{main-thm}, 
		quanched heat kernel fluctuations for the three-dimensional UST. 
	We start with lower fluctuations and then move on to upper fluctuations. 

	Recall that $p_{n}^\ust$ stands for the quenched heat kernel defined in (\ref{hkdef}). 
	
	\begin{theo}\label{HKLower}
		$\prob$-a.s.,
		\begin{equation}\label{hklower}
			\liminf_{n\to\infty} (\log\log{n})^{\frac{\beta-1}{3+\beta}} n^{\frac{3}{3+\beta}}p_{2n}^\ust(0,0)=0.
		\end{equation}
	\end{theo}
	
	\begin{proof}
		By \cite{BK06}*{Theorem 4.1}, we have that 
		\[
			p_{2r|B_\ust(0,r)|}^\ust(0,0)\le \frac{2}{|B_\ust(0,r)|},
		\]
		for any realisation of $\ust$. 
		Let $t_n=n|B_\ust(0,n)|$ and $u_n=n^{-3/\beta}|B_\ust(0,n)|$, then we have 
		\begin{align*}
			p_{2t_n}^\ust(0,0)\le \frac{2}{|B_\ust(0,0)|}&=\frac{2}{t_n^{3/(3+\beta)}n^{-3/(3+\beta)}|B(0,n)|^{\beta/(3+\beta)}}	\\
			&=2t_n^{-3/(3+\beta)}u_n^{-\frac{\beta}{3+\beta}}.
		\end{align*}
		By (\ref{vlowfluc3}), $u_n^{-\frac{\beta}{3+\beta}}>(\log \log n)^{\frac{\beta-1}{3+\beta}}$ infinitely often, almost surely, which completes the proof. 
	\end{proof}
	\begin{remark}
		We proved the existence of some exponent of log-logarithmic term which causes lower heat kernel fluctuation. 
		However, The exponent $\displaystyle \frac{\beta-1}{3+\beta}$ of the log-logarithmic term of $p^\ust_{2n}$ 
			which appears in (\ref{hklower}) is not necesarily a sharp estimate. 
		The critical exponent of the log-logarithmic term has not been obtained even in the critical Galton-Watson tree case. 
	\end{remark}
	
	\begin{prop}
		There exists some constant $\alpha>0$ such that $\prob$-a.s., 
		\begin{equation}\label{hkupper}
			\limsup_{n\to\infty}(\log \log)^{-\alpha}n^{\frac{3}{3+\beta}}p^\ust_{2n}(0,0)=\infty.
		\end{equation}
	\end{prop}
	\begin{proof}
		Let 
		\[
			N_i=(4/c_{10})^{1/3}(\log i)^{1/4},\ \ \ m_i=e^{i^2}/N_i,
		\]
		where $c_{10}$ is as in the statement of Proposition \ref{vvreff}. 
		We follow the construction of a subtree of $\ust$ in $B_\infty(0,N_im_i)$ given in the proof of 
			Proposition \ref{vvreff}. 
		First, let $\gamfty$ be the infinite LERW started at the origin. 
		Then at stage $i\ (i\ge 1)$, we use all vertices in $B_\infty(0,N_im_i)$ and write $\ust_i$ 
			for the obtained tree. 
		Similarly to (\ref{inc2}), we have good separation of scales. 
		By Proposition \ref{vvreff}, conditioned on $\ust_{i-1}$, the event $\widetilde{M}^{N_i}$ 
			occurs with conditional probability greater than $\red{:}$. 
		Now we apply \cite{KuMi08}*{Proposition 3.2} to this $\ust_i$. 
		We set $R=N_i^3m_i^\beta$, $\lambda=1$ and $\vep=c/4(\log N_i)^{300}$ 
			where $c$ is as given in (\ref{reffbound}). 
		Let $r\mathrel{:}[0,\infty]\rightarrow[0,\infty]$ be $r(x)=x$. 
		Then by Proposition \ref{vvreff}, we can set $m$ (which appears in 
			\cite{KuMi08}*{Proposition 3.2}) by $m=c(\log N_i)^{-300}$. 
		Plugging these into (3.6) of \cite{KuMi08}, we have 
		\[
			p^\ust_{2n}(0,0)\ge c'N_i^{-3}m_i^{-3}\ \ \ for\ \ \ n\le \frac{cN_i^6m_i^{3+\beta}}{32(\log N_i)^{300}},
		\]
		for $c>0$ in (\ref{reffbound}) and some constant $c'>0$. 
		Thus, taking $T=\frac{c}{32}N_i^6m_i^{3+\beta}(\log N_i)^{-300}$, it holds that on the event $\widetilde{M}^{N_i}$, we have 
		\[
			T^{\frac{3}{3+\beta}}p^\ust_{2T}(0,0)\ge \left(\frac{c}{64}\right)^{\frac{3}{3+\beta}}c'N_i^{\frac{9-3\beta}{3+\beta}}(\log N_i)^{-\frac{300}{3+\beta}}\ge c''(\log\log T)^{\frac{9-3\beta}{2(3+\beta)}-\vep},
		\]
		for some $c''>0$ and $\vep>0$ which is small. 
		Finally we apply the Borel-Cantelli argument. 
		Since
		\[
			\sum_i c_9\exp\{-c_10N_i^3(\log N_i)\}\ge \sum_i i^{-1}=\infty,
		\]
		by the conditional Borel-Cantelli lemma, we obtain the lower bound (\ref{hkupper}). 
	\end{proof}



\end{document}